\documentclass[12pt]{amsart}
\usepackage{amssymb,amsmath,amstext,amsthm,graphicx,color}
\usepackage{amsmath,amssymb,amscd,amsmath,amstext,mathrsfs,amsfonts,pb-diagram,algorithmic}%

\newif\ifARXIV
\ARXIVtrue 

\usepackage[top=2cm, bottom=2cm, left=2cm, right=2cm]{geometry}
\theoremstyle{plain}
\newtheorem{theorem}{Theorem}

\newtheorem{algorithm}{Algorithm}
\newtheorem{algorithmscheme}{Algorithmic Scheme}

\newtheorem{definition}[theorem]{Definition}

\newtheorem{lemma}[theorem]{Lemma}

\newtheorem{proposition}[theorem]{Proposition}
\newtheorem{remark}[theorem]{Remark}

\numberwithin{equation}{section}

\newcommand{\UsubR}{U_R}

\newcommand{\CHd}{\mathcal{H}_{(d)}}
\newcommand{\CH}[1]{\mathcal{H}_{#1}}

\newcommand{\mfa}{\mathfrak{a}}
\newcommand{\mfb}{\mathfrak{b}}
\newcommand{\Pn}{\P(\C^{n+1})}
\newcommand{\PHd}{\P(\mathcal{H}_{(d)})}

\def\Computeapproxzero{{\sc ShortZero}}

\def\TrackReallyLinearHomotopy{{\sc TrackSegment\_Scheme}}
\def\TrackReallyLinearHomotopypractical{{\sc TrackSegment}}
\def\LowerUpperBoundt{{\sc LUquadratic}}

\def\Signcompute{{\sc Computesign}}

\def\ii{{\mathbf{i}}}

\def\eps{\varepsilon}

\def\C{{\mathbb C}}

\def\Q{{\mathbb Q}}
\def\R{{\mathbb R}}

\def\Z{{\mathbb Z}}
\def\P{{\mathbb P}}
\def\S{{\mathbb S}}

\newcommand{\Gr}{\operatorname{Gr}}
\newcommand{\Real}{\operatorname{Re}}

\newcommand{\mun}{\mu_{\rm norm}}

\def \constantForHomotopyComplexity{79}
\def \ConstantctwooverPtwo{\frac{17}{50000}}

\title{Robust certified numerical homotopy tracking}
\author{Carlos Beltr\'an}
        \thanks{
\ifARXIV
C. Beltr\'an.
\else
C. Beltr\'an (corresponding author). 
\fi
Departamento de Matem\'aticas, Estad\'istica y Computaci\'on, Universidad de Cantabria,
Spain
               ({\tt beltranc@unican.es}). Partially supported by MTM2010-16051, Spanish Ministry of Science
(MICINN)}
\author{Anton Leykin}
        \thanks{Anton Leykin. School of Mathematics, Georgia Tech, Atlanta GA, USA ({\tt leykin@math.gatech.edu}).
\ifARXIV
         Partially supported by NSF grants DMS-0914802 and DMS-1151297
\else
         Partially supported by NSF grants DMS-0914802 and DMS-1151297.
        \\
         \noindent{Communicated by Peter Buergisser}\fi
}

\date{\today}


\begin{document}
\begin{abstract}
We describe, for the first time, a completely rigorous homotopy (path--following) algorithm (in the Turing machine model) to find approximate zeros of systems of polynomial equations. If the coordinates of the input systems and the initial zero are rational our algorithm involves only rational computations and if the homotopy is well posed an approximate zero with integer coordinates of the target system is obtained. The total bit complexity is linear in the length of the path in the condition metric, and polynomial in the logarithm of the maximum of the condition number along the path, and in the size of the input.
\end{abstract}
\maketitle

\section{Introduction}

The research on solving systems of polynomial equations has experienced a rapid development in the last two decades, both from a theoretical and from a practical perspective. Many of the recent advances are based on the study of the general idea of homotopy continuation methods: let $f$ be the system whose solutions we want to find, and let $g$ be another system whose solutions we already know. Then, join $g$ and $f$ with a {\em homotopy}, that is a curve $f_t$ in the vector space of polynomial systems, such that $f_0=g$ and $f_1=f$, and try to follow the curves (homotopy paths) produced as a solution $\zeta_0$ of $g$ is continued along the homotopy to a solution $\zeta_t$ of $f_t$. When $t$ approaches $1$, an approximation of a zero $\zeta_1$ of $f$ is obtained.

In order to describe such a method explicitly, we need two essential ingredients:

\begin{enumerate}
\item A construction of the starting system $g$ and the homotopy path $f_t$, and,
\item \label{item:2} once this path $f_t$ is chosen, a procedure to approximate $\zeta_t$ (for a finite sequence of values of $t$ starting with $t=0$ and ending with $t=1$).
\end{enumerate}

The first of these two ingredients has been intensively studied from many perspectives, mainly using linear homotopy paths, i.e. once $(g,\zeta_0)$ is chosen, $\zeta_0$ a solution of $g$, we just consider the path $f_t=(1-t)g+tf$. In \cite{ShSm94} a particularly simple choice of $(g,\zeta_0)$ was conjectured to be a good candidate for a initial pair, i.e. the complexity of homotopy methods with this starting pair could be polynomial on the average. This is still a challenging open conjecture that has been experimentally confirmed in \cite{BeltranLeykin2009}. In \cite{BePa:FoCM}, \cite{BePa07}, \cite{BeltranPardoFLH} it was proved that randomly chosen pairs $(g,\zeta_0)$ guarantee average polynomial complexity. In \cite{BurgisserCuckerTa} a system whose zeros have coordinates equal to the roots of unity of appropriate degrees was proved to guarantee average quasi--polynomial complexity (polynomial for fixed degree.)

In this paper we deal with the second of the two questions above, restricting ourselves to the case when the homotopy path that is followed is regular, i.e., $\zeta_t$ is a regular zero of $f_t$ for every $t\in [0,1]$.

There exist several software packages which perform the path--following task of item $(\ref{item:2})$ above (here is an incomplete list: Bertini \cite{Bertini}, HOM4PS2 \cite{HOM4PSwww}, NAG4M2~\cite{Leykin:NAG4M2}, and PHCpack \cite{V99}). In general, an initial step $t_0$ is chosen, and a {\em predictor} step (a numerical integration step of the differential equation $d(f_t(\zeta_t))/dt=0$) is made to approximate $\zeta_{t_0}$. A {\em corrector} step (several steps of Newton's method) is then used to get a better approximation of $\zeta_{t_0}$. This process is repeated by choosing $t_1,t_2,\ldots$ until $t=1$ is reached. The software implementations mentioned above achieve spectacular practical results, with huge systems solved in a surprisingly short time. As a drawback, these fast methods make heuristic choices (notably the choice of $t_i$), which may introduce uncertainty in the quality of the solutions they provide: {\em how close to an actual zero is the given output? is the method actually following the path $\zeta_t$ or maybe a path--jumping occurred in the middle?}

In Figure \ref{fig:jumping} we illustrate a {\em path--jumping} phenomenon that may occur when a heuristic predictor--corrector path--tracking procedure is used.
 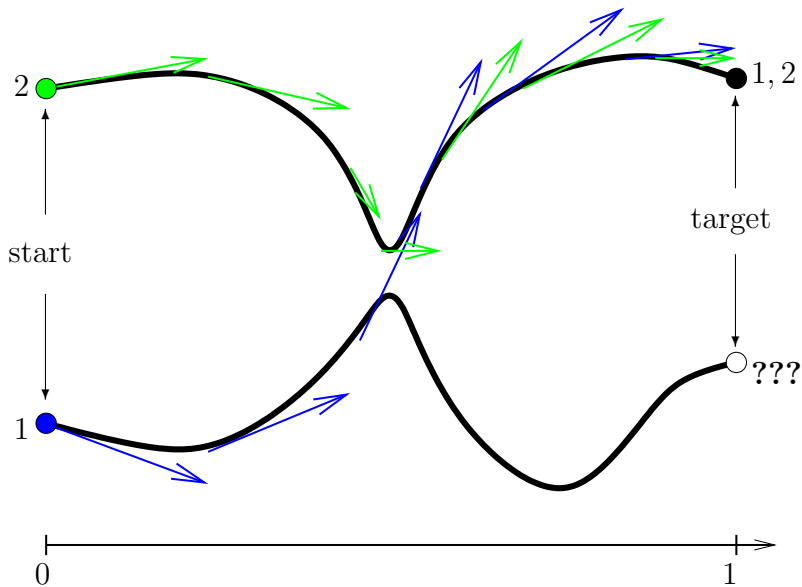
\begin{figure}
 \begin{picture}(300,220)
   \put(10,10){\includegraphics[scale=0.85, keepaspectratio]{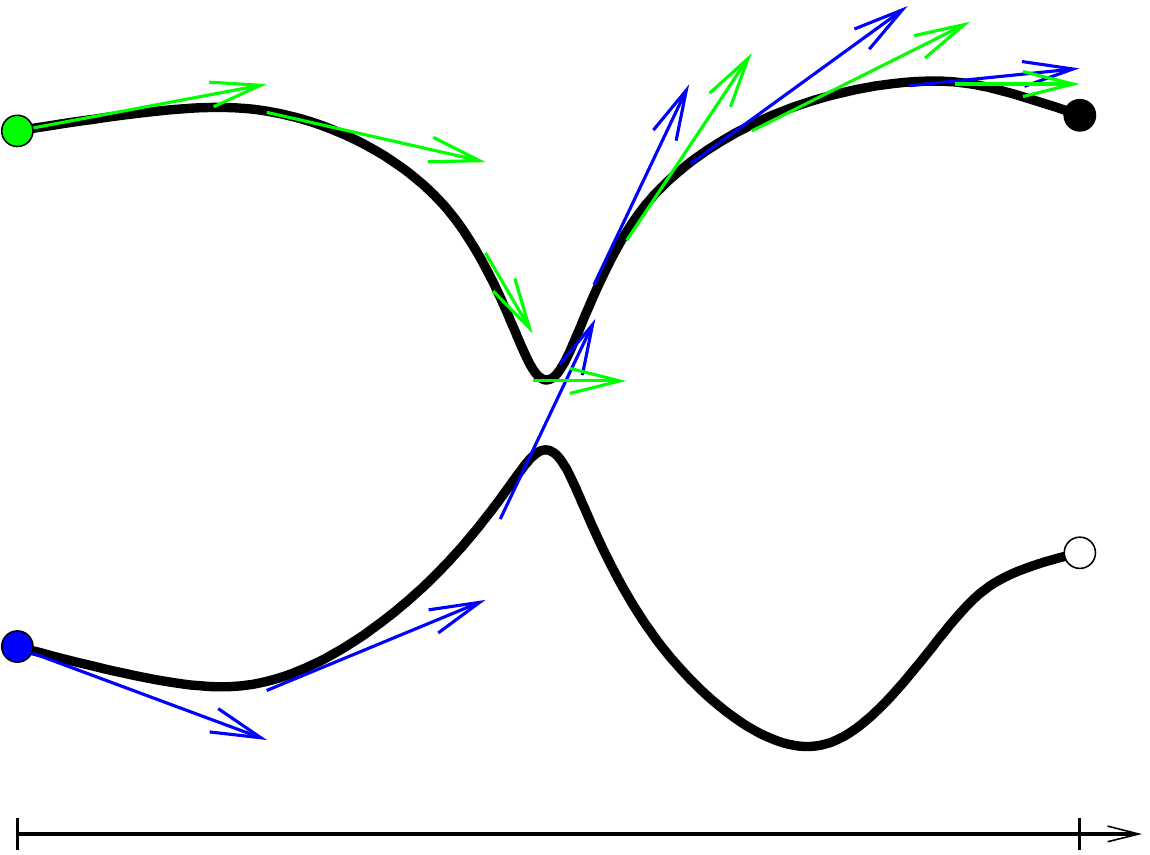}}
   \put(10,0){$0$}
   \put(2,55){$1$}
   \put(2,185){$2$}
   \put(0,122){start}
   \put(14,140){\vector(0,1){40}}
   \put(14,110){\vector(0,-1){40}}
   \put(270,0){$1$}
   \put(281,76){{\bf ???}}
   \put(281,190){$1,2$}
   \put(258,135){target}
   \put(275,150){\vector(0,1){35}}
   \put(275,125){\vector(0,-1){35}}
\end{picture}
 \caption{Path--jumping scenario: we start following the path corresponding to a certain solution of $f_0$ but we may jump to another path in the middle.}
 \label{fig:jumping}
    \end{figure}

In the problems with aim at computing all target solutions, this scenario can be detected simply by observing that two approximation sequences produce approximations to the same regular zero. The Kim--Smale $\alpha$--test from \cite{Kim85}, \cite{Sm86} can be used to make this task rigorously, see \cite{Hauenstein-Sottile:alphaCertified} for an implementation of that test. Once detected, this can be remedied by rerunning the heuristic procedure with tighter tolerances and higher precision of the computation.

However, an analysis of the end solutions would fail to detect the shortcomings of a heuristic method in the scenario with two approximation sequences ``swapping'' two paths as in Figure~\ref{fig:swapping}
 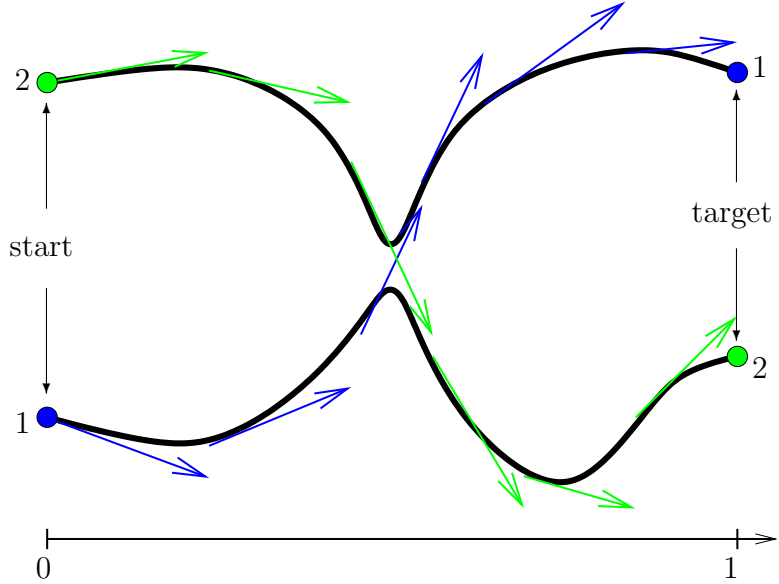
\begin{figure}
 \begin{picture}(300,220)
   \put(10,10){\includegraphics[scale=0.85, keepaspectratio]{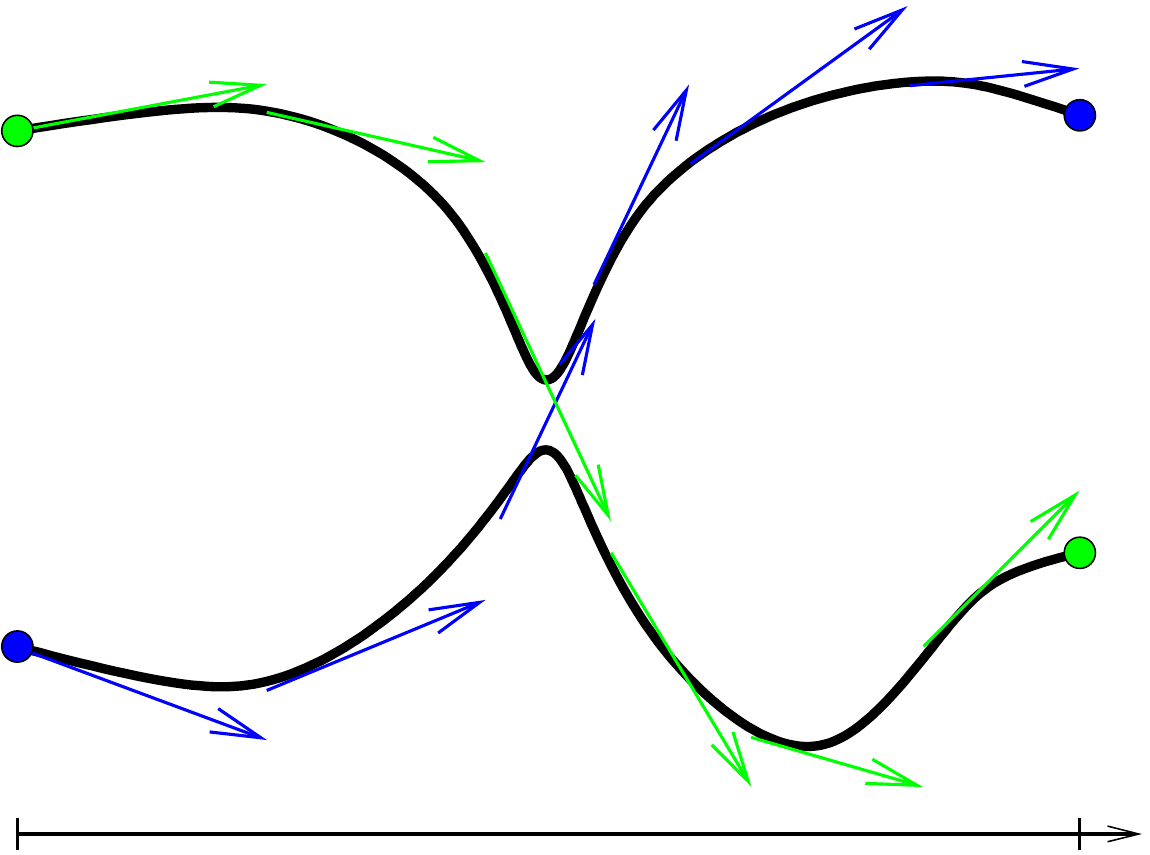}}
   \put(10,0){$0$}
   \put(2,55){$1$}
   \put(2,185){$2$}
   \put(0,122){start}
   \put(14,140){\vector(0,1){40}}
   \put(14,110){\vector(0,-1){40}}
   \put(270,0){$1$}
   \put(281,76){$2$}
   \put(281,190){$1$}
   \put(258,135){target}
   \put(275,150){\vector(0,1){35}}
   \put(275,125){\vector(0,-1){35}}
\end{picture}
 \caption{Path--swapping scenario: two path--jumps occur producing correct (although permuted) target solutions.}\label{fig:swapping}
    \end{figure}

In \cite{ShSm94}, a method which guarantees that path--jumping does not occur was shown for the first time, and its complexity (number of homotopy steps) was bounded above by a quantity depending on the maximum of the so--called {\em condition number} along the path $(f_t,\zeta_t)$, see (\ref{eq:conditionnumber}) below. This result was recently improved in \cite{Shub2007}, changing the maximum of the condition number along the path to the length of the path $(f_t,\zeta_t)$ in the ``condition metric'' (see Section \ref{sec:compandcond} for details). However, the result in \cite{Shub2007} does not fully describe an algorithm, for the explicit choice of the steps $t_i$ is not given. Describing a way to actually choose these steps is a nontrivial task that can be done in several fashions. There are three independent papers doing this job: \cite{Beltran2009NM}, \cite{BurgisserCuckerTa}, \cite{DedieuMalajovichShubTa}. We briefly summarize in table \ref{table:methods} the properties of the algorithms in those papers and of the algorithm in this paper too. The proof of the algorithm in \cite{BurgisserCuckerTa} is probably the shortest of the ones mentioned. As a small drawback, its complexity is not bounded above by the length of the path $(f_t,\zeta_t)$ in the condition metric, but by the integral of the square of the condition number along the path $f_t$, which is in general a (non--dramatically) greater quantity.

\begin{table}
\caption{The existing methods with certified output and complexity analysis for path--tracking. By arithmetic on $\R$ we mean that the BSS computation model \cite{BlShSm} is assumed, that is exact arithmetical operations between real numbers are allowed. C.L. means length in the condition metric.}
\label{table:methods}
\begin{tabular}{|c|c|c|c|}
\hline
Method&Complexity&Arithmetic&Assumptions on $f_t$ and comments\\
\hline
\cite{Shub2007}&C.L.&$\R$&$C^1$ paths $f_t$. Not constructive.\\
\hline
\cite{Beltran2009NM}&C.L.&$\R$&$C^{1+Lip}$ paths $f_t$.\\
\hline
\cite{BurgisserCuckerTa}&$>\approx$C.L. &$\R$&Linear paths.\\
\hline
\cite{DedieuMalajovichShubTa}&C.L.&$\R$&$C^1$ paths $f_t$ with special properties.\\
\hline
This paper&C.L.&$\Q$&Linear paths.\\
\hline

\end{tabular}
\end{table}

All methods in Table \ref{table:methods} are originally designed for systems of homogeneous polynomials, and then an argument like that in \cite{BePa07} is used to produce affine approximate zeros of the original system, if this one was not homogeneous.

Another difference from the heuristic methods described above is that there is no predictor step in these methods; only Newton's method is used (more exactly, projective Newton's method \cite{Sh93} described below): once $t_0$ is chosen, one obtains $z_{t_0}$ as the result of (projective) Newton's method with base system $f_{t_0}$ and base point $z_0=\zeta_0$ (or $z_0$ an approximate zero of $\zeta_0$). The idea is repeated, again, until $t=1$ is reached.

Using the algorithm of \cite{Beltran2009NM}, the main result of \cite{Shub2007} reads as follows.
\begin{theorem}\cite{Shub2007}, \cite{Beltran2009NM}\label{th:numbestepsintro}
Assuming exact numerical computations, and assuming that the path $f_t$ is a great circle in the sphere $\S$ in the space of homogeneous systems, the steps $t_0,t_1,\ldots$ can be chosen in such a way that the homotopy method outlined above produces an approximate zero of $f=f_1$ with associated exact zero $\zeta_1$. The total number of steps is at most a small constant $71 d^{3/2}$ ($d$ the maximum of the degrees of the polynomials in $f$) times $\mathcal{C}_0=\mathcal{C}_0(f_t,\zeta_t)$, that is the length of the path $(f_t,\zeta_t)$ in the so--called condition metric (if that length is infinity, the algorithm may never finish.)
\end{theorem}
Here, the sphere $\S$ is the set of systems $f$ such that $\|f\|=1$ where $\|\cdot\|$ is the Bombieri--Weyl norm described below, and the condition metric is the usual product metric in $\S\times\Pn$, point--wise multiplied by the condition number squared, see (\ref{eq:condlength}) below for a precise definition. The main result of \cite{Shub2007}, \cite{Beltran2009NM} also applies to paths which are not great circles, although the constant $71d^{3/2}$ may then vary.

In recent works \cite{BeltranLeykin2009}, \cite{Leykin:NAG4M2} we described a practical implementation of the ``certified'' method of \cite{Beltran2009NM} and used it in various experiments for estimating the average complexity of homotopy tracking.

As we have already pointed out, Theorem \ref{th:numbestepsintro} (and the other existing algorithms for path--tracking) needs {\em exact numerical computations} on real numbers. More precisely, it assumes computations under the BSS model of computation, \cite{BlShSm}, \cite{BlCuShSm98}. While this assumption permits simplified analysis of the method, it is unrealistic in practice. Our main result here is to remove this assumption: if the input polynomials $g$ and $f$ and the initial approximate zero $z_0$ are given in (Gaussian) rational coordinates, then all the computations can be made over the rationals. We center our attention in linear paths, that is paths of the form $f_t=(1-t)g+tf$ where $t\in[0,1]$. It is a natural fact that the condition number plays an important role in the translation of the real--number arithmetic results to rational arithmetic results. We will see that this just produces an extra factor (the logarithm of the maximum of the condition number along the path) in the complexity bounds. Our main result thus reads as follows.

\begin{theorem}[Main]\label{th:mainintro}
Assuming that $g$, $f$ and $z_0$ are given in rational coordinates, and assuming the extra hypotheses (\ref{eq:hyp}) below, the homotopy method can be designed (see \TrackReallyLinearHomotopypractical{} below) to produce an approximate zero $z_*$ with integer coordinates of $f=f_1$ with associated exact zero $\zeta_1$. The total number of steps is a small constant $C \sqrt{n}d^{3/2}$ ($d$ the maximum of the degrees of the polynomials in $f$, and $n+1$ the number of variables) times $\mathcal{C}_0$. The bit complexity of the algorithm is linear in $\mathcal{C}_0$ and polynomial in the following quantities:
\begin{itemize}
\item $n,S,d,h$, where $S$ is the number of nonzero monomials in the dense expansions of $f$ and $g$ and $h$ is a bound for the bit length of the rational numbers appearing in the description of $f,g,z_0$,
\item $\log_2(\max_{t\in[0,1]}\{\mu(f_t,\zeta_t)\})$, and
\item The quantity
\[
\log_2\left(\max\left(1,\frac{\|f-g\|}{\min\{\|f_t\|:0\leq t\leq 1\}}\right)\right),
\]
\end{itemize}
Here, $\|\cdot\|$ is Bombieri--Weyl norm (we recall its definition in Section \ref{sec:Hd}) and $\mu(f_t,\zeta_t)$ is condition number at the pair $(f_t,\zeta_t)$, see (\ref{eq:conditionnumber}) below. The bit size of the integer numbers in the coordinates of $z_*$ is at most $O(\log_2(n)+\log_2\mu(f,\zeta_1))$.
\end{theorem}
The extra hypotheses that will be specified in (\ref{eq:hyp}) is, in words, that the angle between $g$ and $f-g$ is not too close to $\pi$, that is that the segment $\{f_t:t\in[0,1]\}$ does not point too straight--forward to the origin.

The reader may note that, once this method is proved to work and programmed, it provides a status of mathematical proof to the path--following procedure. Moreover, its complexity does not differ much from the method of Theorem \ref{th:numbestepsintro} and the size of the integer numbers involved is controlled.

In particular, the method can be used to give rigorous proofs to the results obtained via monodromy computation by algorithms in \cite{Leykin-Sottile:HoG}, \cite{Bates-Peterson-Sommese-Wampler:genus}. Note that the implementations of the $\alpha$-test carried out in \cite{Hauenstein-Sottile:alphaCertified}, \cite{Malajovich:PSS} would not be sufficient for the above applications even in the situations where the zero count is known and the $\alpha$-test is capable of certifying the exact expected number of distinct zeroes at the ends of the homotopy paths. This is due to a potential multiple path--jumping (which we just call path--swapping) resulting in a wrong permutation of zeroes produced by tracking monodromy loops. In other words, as we have already pointed out, while in certain cases the $\alpha$-test for the end solutions can resolve the scenario of Figure~\ref{fig:jumping}, it is powerless in the scenario of Figure~\ref{fig:swapping}.

\subsection{Previous works and historical remarks}

Smale \cite{Sm81} proved the first results on the average complexity of polynomial zero--finding using Newton's method. The problem of solving systems of polynomial equations became later one of the cornerstones of complexity theory in the BSS model. In the first paragraph of \cite{Sm81} we read {\em ``Also this work has the effect of helping bring the discrete mathematics of complexity theory of computer science closer to classical calculus and geometry''}. This paper is written with the intention of making another step in that direction. Theorem \ref{th:mainintro} establishes a strong link between the BSS model of computation and the classical Turing machine model, and we believe that this link strengthens both models at the same time. The Turing machine model is proved to accomplish a difficult task until now reserved to numerical solvers. An algorithm designed and analyzed in the BSS model is successfully translated (including its complexity analysis) to the classical model by carefully studying its condition number. Adding up, in the case of path--tracking methods for polynomial system solving,
\[
\text{BSS model }+\text{ condition number analysis } = \text{Turing machine model},
\]
the equality being strong in the sense that the complexity of the discrete algorithm is similar to the complexity of the BSS algorithm. This ``translation'' of computational models can probably be done for many of the algorithms originally designed in the BSS model. The reader may find many of our techniques useful for such a task.

A natural precedent to our work is the algorithm in Malajovich's Ph.D. Theses \cite{Ma93} (see also \cite{Ma94}, \cite{Ma2000}) where a homotopy method for polynomial systems with coefficients in $\Z[\ii]$ is presented. The algorithm in \cite{Ma93} is certified under reasonable assumptions for $\eps$--machines (i.e. floating point machines), in which some intermediate computations are made. Its total complexity is bounded by a number which does not explicitly depend on the condition number of the systems found in the path. In \cite{Ma93}, \cite{Ma2000} there are some ``gap theorems'' which give universal bounds for the value of the condition number in the case that the coefficients of the polynomials are integers (and assuming that the solutions are not singular). The biggest advantage of that approach is that a global complexity bound is found. By ``global'' we mean that it is valid for solving every {\em generic} system; as a drawback, that complexity bound is exponential on the bit size of the entries.

In \cite[Cor. 2.1]{CaHaMoPa}, a universal upper bound for the bit size of rational approximate zeros of smooth zeros of systems with integer coefficients and degrees at most $2$ is given. In general, one expects the condition number to be much smaller than in the worst case (see \cite{ShSm93}, \cite{CaMoPaSa} for results in this direction), and hence the bit size of the output of our algorithm will in general be much smaller than the upper bound of \cite[Cor. 2.1]{CaHaMoPa}.

One alternative to the rational arithmetic approach of this paper to the certification of homotopies could be using interval arithmetic. For example, in a more general setting, \cite{Kearfott} proposes step control by means of isolating a homotopy path in a box around an approximation of a point on the path. While the implementation of \cite{Kearfott} does not provide certification, in principle, interval arithmetic can be used in an attempt to certify homotopy tracking using similar isolation ideas (see \cite{vdHoeven} for an ongoing work in this direction).

\subsection{Acknowledgments}
In earlier stages of the ideas behind this work, we maintained many related conversations with Clement Pernet; thanks go to him for helpful discussions and comments. We also thank Gregorio Malajovich, Luis Miguel Pardo and Michael Shub for their questions and answers. Our beloved friend and colleague Jean Pierre Dedieu also inspired us in many occasions. The second author thanks Institut Mittag-Leffler for hosting him in the Spring semester of 2011. A part of this work was done while we were participating in several workshops related to Foundations of Computational Mathematics in the Fields Institute. We thank this institution for its kind support.
\section{Technical background}

\subsection{The vector space of polynomial systems}\label{sec:Hd}
As mentioned above, we will center our attention on homogeneous systems of equations. For fixed $n\geq1$ and an integer $l\geq1$, let $\CH{l}\subseteq\C[X_0,\ldots,X_n]$ be the vector space of degree $l$ homogeneous polynomials with unknowns $X_0,\ldots,X_n$. As in \cite{ShSm93b}, we consider Bombieri--Weyl's Hermitian product $\langle\cdot,\cdot\rangle_{\CH{l}}$ which preserves the orthogonality of different monomials and satisfies
\[
\langle X_0^{\alpha_0}\cdots X_n^{\alpha_n},X_0^{\alpha_0}\cdots X_n^{\alpha_n}\rangle_{\CH{l}}=\frac{ \alpha_1!\cdots\alpha_n!}{l!}.
\]
For integers $d_i\geq1$, $1\leq i\leq n$ and an $n$--tuple $(d)=(d_1,\ldots,d_n)$, we denote by $\CHd$ the vector space of systems of $n$ homogeneous polynomials of respective degrees $d_1,\ldots,d_n$ in unknowns $X_0,\ldots,X_n$. That is,
\[
\CHd=\CH{d_1}\times\cdots\times\CH{d_n}.
\]
The Hermitian product in $\CHd$ is then
\[
\langle f,g\rangle=\langle f_1,g_1\rangle_{\CH{d_1}}+\cdots+\langle f_n,g_n\rangle_{\CH{d_n}}
\]
where $f=(f_1,\ldots,f_n),\;g=(g_1,\ldots,g_n)\in\CHd$. We also define
\[
\|f\|=\langle f,f\rangle^{1/2}.
\]
The Hermitian product (norm) given by $\langle\cdot,\cdot\rangle$ ($\|\cdot\|$) in $\CHd$ is also called Bombieri--Weyl product (norm). It has a number of nice properties such as invariance under composition with a unitary change of coordinates, see for example \cite[Section 12.1]{BlCuShSm98}. We denote
\[
\S=\{f\in\CHd:\|f\|=1\}.
\]

Our main algorithm (Algorithm \ref{alg: certified homotopy really linear} below) will need to compute $\langle\cdot,\cdot\rangle$ and $\|\cdot\|^2$. This can obviously be done over the (complex) rationals $\Q[\ii]$, if the polynomials involved have coordinates in $\Q[\ii]$.
\begin{remark}
The extra hypothesis Theorem \ref{th:mainintro} needs is
\begin{equation}\label{eq:hyp}
-L_0\|g\|\,\|f-g\|\leq \Real{\langle g,f-g\rangle}\leq \|g\|\,\|f-g\|,
\end{equation}
where $L_0=1-10^{-3}$. While this hypothesis is not actually needed for the method to work (an output will equally be obtained if the hypothesis is not satisfied), it does simplify the complexity analysis of the algorithm significantly. Note that (\ref{eq:hyp}) means that the angle between the two vectors $g$ and $f-g$ in the space of polynomial systems is not too close to $\pi$. This is satisfied by most reasonable choices of paths $f_t$.
\end{remark}
\subsection{Projective Newton method}
We now describe the projective Newton method of \cite{Sh93}. Let $f\in\CHd$ and $z\in\Pn$. Then,
\[
N_\P(f)(z)=z-\left(Df(z)\mid_{z^\perp}\right)^{-1}f(z),
\]
where $Df(z)$ is the $n\times (n+1)$ Jacobian matrix of $f$ at $z\in\Pn$, and
\[
Df(z)\mid_{z^\perp}
\]
is the restriction of the linear operator defined by $Df(z):\C^{n+1}\rightarrow\C^n$ to the orthogonal complement $z^\perp$ of $z$. The reader may check that $N_\P(f)(\lambda z)=\lambda N_\P(f)(z)$, namely $N_\P(f)$ is a well--defined projective operator as long as the linear map $Df(z)\mid_{z^\perp}$ has an inverse. An equivalent expression, better suited for computations, is
\[
N_\P(f)(z)=z-\binom{Df(z)}{z^*}^{-1}\binom{f(z)}{0},
\]
where $z^*$ is the conjugate transpose of $z$. We denote by $N_\P(f)^l(z)$ the result of $l$ consecutive applications of $N_\P(f)$ on initial point $z$.

In general, one cannot expect that the solutions of systems of polynomials have rational coordinates. The goal of solvers is thus to produce rational points which are ``close'' in some sense to actual zeros. Following the approach of \cite{Sm86}, \cite{ShSm93b} we will consider that a point is an ``approximate zero'' of a system of equations if it is in the strong (quadratic)  basin of attraction of the projective Newton method. Namely:
\begin{definition}\label{def:projzppzero}
We say that $z\in\Pn$ is an approximate zero of $f\in\CHd$ with associated (exact) zero $\zeta\in\Pn$ if $N_\P(f)^l(z)$ is defined for all $l\geq0$ and
\[
d_R(N_\P(f)^l(z),\zeta)\leq\frac{d_R(z,\zeta)}{2^{2^{l}-1}},\;\;\;\;l\geq0.
\]
\end{definition}
Here $d_R$ is the Riemann distance in $\Pn$, namely
\[
d_R(z,z')=\arccos\frac{|\langle z,z'\rangle|}{\|z\|\,\|z'\|}\in[0,\pi/2],
\]
where $\langle \cdot,\cdot\rangle$ and $\|\cdot\|$ are the usual Hermitian product and norm in $\C^{n+1}$. Note that $d_R(z,z')$ is the length of the shortest $C^1$ curve with extremes $z,z'\in\Pn$, when $\Pn$ is endowed with the usual Hermitian structure (see for example \cite[Page 226]{BlCuShSm98}.)
\subsection{The condition number}
The condition number at $(f,z)\in\CHd\times\Pn$ introduced in \cite{ShSm93b} is defined as follows:
\begin{equation}\label{eq:conditionnumber}
\mu(f,z)=\|f\|\,\|(Df(z)\mid_{\,z^\perp})^{-1} \mbox{Diag}(\|z\|^{d_i-1}d_i^{1/2})\|,
\end{equation}
or $\mu(f,z)=\infty$ if $Df(z)\mid_{\,z^\perp}$ is not invertible. Here, $\|f\|$ is the Bombieri-Weyl norm of $f$ and the second norm in the product is the operator norm of that linear operator. Note that $\mu(f,z)$ is, up to some normalizing factors, essentially equal to the operator norm of the inverse of the Jacobian  $Df(z)$, restricted to the orthogonal complement of $z$. Sometimes $\mu$ is denoted $\mun$ or $\mu_{\rm proj}$, but we keep the simplest notation here. One of the main properties of $\mu$ is\footnote{This property inspired its definition, see \cite{ShSm93b}.} that it bounds the norm of the implicit function of the mapping $(f,z)\mapsto f(z)$. In other words, following \cite[Sec. 12.3 and 12.4]{BlCuShSm98}:
\begin{lemma}\label{lem:muboundsimplicit}
Let $(g,\zeta_0)\in\CHd\times \Pn$ be such that $g(\zeta_0)=0$, and $\mu(g,\zeta_0)<\infty$. Let $f_t,t\in[0,\varepsilon)$ be a $C^1$ curve in $\CHd$, $f_0=g$. Then, for sufficiently small $t<\varepsilon$, $\zeta_0$ can be continued to a zero $\zeta_t$ of $f_t$, that is there exists a $C^1$ curve $t\mapsto\zeta(t)\subseteq\Pn$ such that $\zeta(0)=\zeta_0$ and, denoting $\zeta(t)=\zeta_t$, we have $f_t(\zeta_t)=0$ for every sufficiently small $t$. Moreover, the tangent vectors satisfy:
\[
\|\dot\zeta_0\|\leq\mu(g,\zeta_0)\|\dot f_0\|.
\]
\end{lemma}
We will also use a variation of this condition number, namely $\chi_1$ in equation (\ref{equ:chi1}) below.

The following result is a version of Smale's $\gamma$--theorem (cf. \cite{Sm86}), and follows from the study of the condition number in \cite{ShSm93b}, \cite{Shub2007}.

\begin{proposition}\label{prop:aptproj}\cite[Lemma 6]{Beltran2009NM}
Let $\zeta\in\Pn$ be a zero of $f\in\CHd$ and let $z\in\Pn$ be such that
\[
d_R(z,\zeta)\leq \frac{u_0}{d^{3/2}\mu(f,\zeta)},\;\;\;\;\text{ where $u_0=0.17586$}.
\]
Then $z$ is an approximate zero of $f$ with associated zero $\zeta$.
\end{proposition}

\subsection{Complexity and the condition metric}\label{sec:compandcond}
According to \cite{Shub2007}, the complexity (dominated by the number of Newton steps or number of while loops) of an algorithm performing the homotopy method should depend on the so--called condition length of the homotopy path. Given a path $(h_t,\zeta_t)$, $t_0\leq t\leq t_1$ where $\zeta_t$ is a zero of $h_t$ and $h_t\in\P(\CHd)$, $\zeta_t\in\Pn$, the length of the path $(h_t,\zeta_t)$ in $\P(\CHd)\times\Pn$ is given by the integral
\begin{align*}
&\int_{t_0}^{t_1}\left\|\frac{d}{dt}(h_t,\zeta_t)\right\|_{T_{(h_t,\zeta_t)}\left(\PHd\times\Pn\right)} 
dt \\
&\ = \ \int_{t_0}^{t_1}\sqrt{\|\dot{h}_t\|_{T_{h_t}\PHd}^2+\|\dot{\zeta}_t\|_{T_{\zeta_t}\Pn}^2}\;dt
\end{align*}
One must here understand $\dot{h}_t$ and $\dot\zeta_t$ as tangent vectors in $T_{h_t}\PHd$ and $T_{\zeta_t}\Pn$ respectively. If $h_t$ and $\zeta_t$ are given in coordinates, this means:
\[
\|\dot{h}_t\|_{T_{h_t}\PHd}^2=\frac{\|\dot{h}_t\|^2}{\|h_t\|^2}-\frac{|\langle \dot{h}_t,h_t\rangle|^2}{\|h_t\|^4},
\]
\[
\|\dot{\zeta}_t\|_{T_{\zeta_t}\Pn}^2=\frac{\|\dot{\zeta}_t\|^2}{\|\zeta_t\|^2}-\frac{|\langle \dot{\zeta}_t,\zeta_t\rangle|^2}{\|\zeta_t\|^4},
\]

Now, the condition length (or length in the condition metric) of the same path is defined in \cite{Shub2007} as
\[
\int_{t_0}^{t_1}\mu(h_t,\zeta_t)\sqrt{\|\dot{h}_t\|_{T_{h_t}\PHd}^2+\|\dot{\zeta}_t\|_{T_{\zeta_t}\Pn}^2}\;dt.
\]
Note that, from (\ref{eq:conditionnumber}), $\mu(h_t,\zeta_t)$ only depends on the projective classes of $h_t$ and $\zeta_t$ and this last integral is thus well defined.

Now, given a path $(f_t,\zeta_t)\in\CHd\times\Pn$ where $f_t(\zeta_t)=0$, we define its condition length as
\begin{equation}\label{eq:condlength}
\mathcal{C}_0=\mathcal{C}_0(f_t,\zeta_t)=\int_{t_0}^{t_1}\mu(f_t,\zeta_t) \sqrt{\frac{\|\dot f_t\|^2}{\|f_t\|^2}-\frac{\Real{(\langle \dot{f}_t,f_t\rangle)}^2}{\|f_t\|^4}+\|\dot\zeta_t\|_{T_{\zeta_t}\Pn}^2}\,dt.
\end{equation}
That is, $\mathcal{C}_0(f_t,\zeta_t)$ is the length in the condition metric on $\S\times\Pn$ of the path obtained by projecting $(f_t,\zeta_t)\in\CHd\times\Pn$ on $\S\times\Pn$. The reason of this change is that segments in $\CHd$ project nicely on the sphere $\S$ (indeed, they project onto pieces of great circles in $\S$), but they do not project so nicely on $\P(\CHd)$. This makes our analysis easier and, presumably, has little effect in the complexity bounds.

Note that, if $f_t\in\CHd$ is the horizontal lift of $h_t\in\P(\CHd)$ then the condition lengths of $(f_t,\zeta_t)$ and $(h_t,\zeta_t)$ coincide. In the general case, however, the condition length of $(f_t,\zeta_t)$ is greater than that of $(\pi_{\P(\CHd)}(f_t),\zeta_t)$, because in general $\Real{(\langle\dot f_t,f_t\rangle)}^2\leq|\langle \dot f_t,f_t\rangle|^2$.

\section{A robust homotopy step}
In this section we set up the backbone of our main algorithm: how to correctly choose a homotopy step. We do this by stating a theorem that, given sufficiently close polynomial systems $g,f\in\CHd$ and an approximate zero $z_0$ of $g$ associated to an actual zero $\zeta_0$ of $g$, guarantees that $\zeta_0$ can be continued to a zero $\zeta_1$ of $f$, and moreover a projective point sufficiently close (in a sense that we will precisely determine) to $N_\P(f)(z_0)$ is an approximate zero of $f$ with associated zero $\zeta_1$. There are some precedents to this result in \cite{Ma93} but our theorem is needed to get the sharp complexity bound of \cite{Shub2007}.

\subsection{Some constants}\label{subsec:constants}
As it is common in the explicit description of many numerical analysis algorithms, we will need to use some constants, that need to be described explicitly because they intervene in the definition of the algorithm. We will use a free parameter $1/2<\delta<1$ that will be set to $3/4$ in our implementation of the algorithm. The rest of the constants are:
\[
u_0=0.17586 \text{ (the constant from Proposition \ref{prop:aptproj}),}
\]
\[
P=\sqrt{2}+\sqrt{4+5/8},\qquad a=a_\delta= \frac{(2\delta-1)u_0}{\sqrt{2}+2\delta u_0} <\frac{1}{\sqrt{2}},
\]
\begin{equation}\label{eq:cA2b}
c'=c'_\delta=1-\left(1-a\right)^ { \frac { P } { \sqrt { 2 }}}<1.
\end{equation}
Finally, let $c$ be any number satisfying
\begin{equation}\label{eq:cb}
c\leq c_\delta=\frac{(1-\sqrt{2}u_0/2)^{\sqrt{2}}}{1+\sqrt{2}u_0/2}\,c'.
\end{equation}
Only the value of $c^2/(2P^2)$ will appear in the description of our main algorithm (see Algorithm \TrackReallyLinearHomotopypractical{} below.) Hence, one can choose any value of $(c^2/2P^2)\in\Q$ such that
\[
\frac{c^2}{2P^2}\leq\frac{c_{3/4}^2}{2P^2}=0.00034412...
\]
In our algorithm, we will choose $\ConstantctwooverPtwo=0.00034$.
The reader may check that the following holds.
\begin{equation}\label{eq:inequalityb}
\frac{c'}{P(1-a)}\leq \frac{3\delta u_0}{2},\quad \frac{c'}{P}\leq\frac{3a}{2\sqrt{2}}.
\end{equation}

\subsection{A version of the condition number}
The condition number $\mu(f,\zeta)$ of (\ref{eq:conditionnumber}) above can be computed using a more amenable expression if $f(\zeta)=0$.

Let $g,\dot g\in\CHd$ be two polynomial systems and let $z\in\Pn$. Let $\chi_1=\chi_1(g,z)$, $\chi_2=\chi_2(g,\dot g,z)$ and $\varphi=\varphi(g,\dot g,z)$ be defined by
\begin{equation}\label{equ:chi1}
\chi_{1}=\left\|\binom{Dg({z})}{{z}^*}^{-1}
\begin{pmatrix}\sqrt{d_1}\|g\|\|{z}\|^{d_1-1}& & & \\
& \ddots & &\\
& &\sqrt{d_n}\|g\|\|{z}\|^{d_n-1}\\
& & &\|{z}\|\end{pmatrix}
\right\|,
\end{equation}
\begin{equation}\label{equ:chi2}
\chi_{2}=\left(\|\dot g\|^2+\frac{\|g\|^2}{\|z\|^2}\left\|\binom{Dg({z})}{{z}^*}^{-1}
\binom{\dot
g({z})}{0}
\right\|^2\right)^{1/2},
\end{equation}
\begin{equation}\label{equ:phi}
\varphi=\chi_{1}\chi_{2}.
\end{equation}
Note that these formulas do not depend on the representative of $z$ and thus are well defined. Their value is also invariant under multiplication of $g$ by a non--zero complex number $\lambda\in\C$.

It was noted in \cite[eq. (2.2)]{Beltran2009NM} that if $t\mapsto (f_t,\zeta_t)\subseteq\S\times\Pn$ is a $C^1$ curve such that $f_t(\zeta_t)=0$, then
\begin{equation}\label{eq:22NM}
\chi_1(f_t,\zeta_t)=\mu(f_t,\zeta_t),\quad\varphi(f_t,\dot f_t,\zeta_t)=\mu(f_t,\zeta_t)\|(\dot f_t,\dot\zeta_t)\|,\quad \forall\;t.
\end{equation}
Thus, $\chi_1(f,z)$ is a version of the condition number $\mu(f,z)$ (equal if $f(z)=0$) and $\varphi$ is, if $(f_t,\zeta_t)\subseteq\S\times\Pn$ and $f_t(\zeta_t)=0$, the quantity inside of the integral defining $\mathcal{C}_0$ in (\ref{eq:condlength}).

From Lemma \ref{lem:muboundsimplicit}, with the notation of (\ref{eq:22NM}) we have:
\begin{equation}\label{eq:varphihatvsmuprev}
\varphi(f_0,\dot f_0,\zeta_0)\leq\mu(f_0,\zeta_0)\|\dot f_0\|\sqrt{1+\mu(f,\zeta_0)^2}.
\end{equation}
The reason to use $\chi_1$ instead of $\mu$ and $\varphi$ instead of simply the term inside the integral defining $\mathcal{C}_0$ is that the rates of change of $\chi_1$ and $\varphi$ are easy to analyze, see lemmas \ref{lem:homoto1b} and \ref{lem:homoto2} in Section \ref{sec:proofth2}.

\subsection{A robust homotopy step}
We now state our main technical tool, which is a more detailed and complete version of Lemma \ref{lem:muboundsimplicit} about continuation of zeros, designed to answer the following questions
\begin{itemize}
\item how long can a zero of $g$ be continued when $g$ is moved?
\item how do $\chi_1$ and $\varphi$ vary in this process?
\item if an approximate zero of $g$ is given, how long will it still be an approximate zero as $g$ is moved?
\end{itemize}
We will use the constants defined in Section \ref{subsec:constants}. Given two systems $f,g\in\CHd$, we consider the Riemannian distance in the sphere $\S$ from $g/\|g\|$ to $f/\|f\|$, that is:
\[
d_\S\left(\frac{g}{\|g\|},\frac{f}{\|f\|}\right)=\arccos\frac{\Real{\langle g,f\rangle}}{\|g\|\,\|f\|}.
\]
\begin{theorem}\label{th:onestep}
Let $g,f\in\CHd$ be two systems of polynomial equations such that $g\neq \lambda f$ $\forall \lambda\in\R$. Let $z_0$ be an approximate zero of $g$ satisfying
\begin{equation}\label{eq:hypz0}
d_R(z_0,\zeta_0)\leq \frac{u_0}{2d^{3/2}\mu(g,\zeta_0)}
\end{equation}
for some exact zero $\zeta_0$ of $g$. Let
\[
\dot g= \frac{\|g\|^{2}f-\Real(\langle f,g\rangle) g}{\|g\|\sqrt{\|f\|^2\|g\|^2-\Real(\langle f,g\rangle)^2}}.
\]
That is, $\dot g$ is the derivative at $0$ of the arc--length parametrized short portion of the great circle in $\S$, from $g/\|g\|$ to $f/\|f\|$. Let
\[
\chi_1=\chi_1(g,z_0),\quad \chi_2=\chi_2(g,\dot g,z_0),\quad \varphi=\varphi(g,\dot g,z_0).
\]
Assume that
\begin{equation}\label{eq:distancefgteor}
d_\S\left(\frac{g}{\|g\|},\frac{f}{\|f\|}\right)\leq \frac{c}{Pd^{3/2}\varphi}.
\end{equation}
Then,
\begin{enumerate}
\item $\zeta_0$ can be continued following the straight line homotopy
\begin{equation}\label{equ: linear homotopy}
f_t=(1-t)g+tf
\end{equation}
to a zero $\zeta_t$ of $f_t$, namely, there exists a $C^1$ curve $t\mapsto \zeta(t)$ such that $\zeta(0)=\zeta_0$ and, denoting $\zeta(t)=\zeta_t$, we have $f_t(\zeta_t)=0$ for $t\in[0,1]$.
\item We have the following inequality:
\begin{equation}\label{eq:varphivsmu}
\varphi\leq\frac{\sqrt{2}\mu(g,\zeta_0)^2}{(1-\sqrt{2}u_0/2)^{1+\sqrt{2}}},
\end{equation}
\item The condition length $\mathcal{C}_0(f_t,\zeta_t)$ of the path $(f_t,\zeta_t)$ as defined in (\ref{eq:condlength}) is essentially equal to $\varphi\,d_\S\left(\frac{f}{\|f\|},\frac{g}{\|g\|}\right)$. More exactly:
\begin{equation}\label{eq:teorconc2}
\frac{(1-\sqrt{2}u_0/2)^{1+\sqrt{2}}\ln(1+c')}{c'}\leq\frac{\mathcal{C}_0(f_t,\zeta_t)}{\varphi d_\S\left(\frac{f}{\|f\|},\frac{g}{\|g\|}\right)}\leq \frac{1+\sqrt{2}u_0/2}{(1-\sqrt{2}u_0/2)^{\sqrt{2}}}.
\end{equation}
\item For every $\tilde{z}\in\Pn$ such that
\begin{equation}\label{eq:appztildeteor}
d_R(\tilde{z},N_\P(f)(z_0))\leq \frac{(1-\delta)u_0}{2d^{3/2}(1+3\,\delta u_0/2) \chi_{1}}
\end{equation}
we have that
\begin{equation}\label{eq:teorconclussion}
d_R(\tilde{z},\zeta_1)\leq \frac{u_0}{2d^{3/2}\mu(f,\zeta_1)}.
\end{equation}
In particular, $\tilde{z}$ is an approximate zero of $f$ with associated zero $\zeta_1$.
\end{enumerate}
\end{theorem}
The proof of this theorem is a long and tedious computation. We delay it till Section \ref{sec:proofth2}.

\section{A schematic description of the robust linear homotopy method}
In this section we describe an algorithmic scheme for the linear homotopy method. The procedure in this section is not quite an algorithm, because we do not specify how to perform some of the tasks it requires. We will however prove that {\em any} actual algorithm designed to fit into the scheme of this section has certified output and the number of iterations it performs is essentially bounded by the condition length $\mathcal{C}_0$.

Let $g,f$ be two non--collinear systems, that is $g\neq\lambda f$ for every $\lambda\in\R$. Let $f_t$ be defined by (\ref{equ: linear homotopy}), so $f_0=g$, $f_1=f$. Let $\delta\in\Q$, $R\in\sqrt{\Q}$ with $1/2<\delta<1$ and $R\geq\sqrt{2}$ be two arbitrary constants \footnote{If $R\geq\sqrt{2}$ then values $t_i$ as described in Algorithm \ref{alg: certified homotopy really linear} exist. They also exist for smaller values of $R>1$ like $R=1.0003$ but taking $R\geq\sqrt{2}$ will make our formulas look prettier. This assumption does not affect much to the running time of the algorithm. We will only use $R^2$ in the algorithm. Hence, we can take $R\in\sqrt{\Q}$.}.

\begin{algorithmscheme} \label{alg: certified homotopy really linear}$z_*=\mbox{\TrackReallyLinearHomotopy{}}(f,g,z_0)$
\begin{algorithmic}[1]
\REQUIRE $f,g\in\CHd$ non--collinear with coefficients in $\Q[i]$;\\ $z_0\in\Q[i]^{n+1}$ is an approximate zero of $g$ satisfying (\ref{eq:hypz0}).
\ENSURE $z_*\in\Z[i]^{n+1}$ is an approximate zero of $f$ associated to the end of the homotopy path starting at the zero of $g$
associated to $z_0$
and defined by the homotopy (\ref{equ: linear homotopy}).
\STATE $i \leftarrow 0$; $s_i\leftarrow0$.
\WHILE {$s_i \neq 1$}
\STATE $g_i\leftarrow f_{s_i}$.
\STATE Let
\[
\dot g_i= \frac{\|g_i\|^{2}f-\Real(\langle f,g_i\rangle) g_i}{\|g_i\|\sqrt{\|f\|^2\|g_i\|^2-\Real(\langle f,g_i\rangle)^2}}.
\]
Let $\chi_{i,1}=\chi_{i,1}(g_i,z_i)$, $\chi_{i,2}=\chi_{i,2}(g_i,\dot g_i,z_i)$ and $\varphi_i=\varphi_i(g_i,\dot g_i,z_i)$ as defined in (\ref{equ:chi1}), (\ref{equ:chi2}) and (\ref{equ:phi}).
\label{step:phi}
\STATE Let $t_i$ be any positive number such that\
\begin{equation}\label{eq:newstepeqn}
L\leq \frac{\|g_i\|^2+t_i\Real{\langle g_i,f-g\rangle}}{\|g_i\|\sqrt{\|g_i\|^2+2t_i\Real{\langle g_i,f-g\rangle}+t_i^2\|f-g\|^2}}\leq  \UsubR
\end{equation}
where
\begin{equation}\label{eq:LU}
L=1-\frac{c^2}{2P^2d^3\varphi_i^2}+\frac{c^4}{24P^4d^6\varphi_i^4},\qquad \UsubR=1-\frac{c^2}{2R^2P^2d^3\varphi_i^2}.
\end{equation}
If such $t_i$ does not exist, let $t_i\leftarrow1$ (which will imply that the while loop finishes in this step).
\IF{$t_i>1-s_i$}
\STATE $t_i\leftarrow 1-s_i$.
\ENDIF
\STATE $s_{i+1} \leftarrow s_i + t_i $;
\STATE $\displaystyle \eps\leftarrow \frac{(1-\delta)^2u_0^2}{ 4d^{3}(1+3\delta u_0/2)^2 \chi_{i,1}^2 }$.
\STATE $z_{i+1}\leftarrow N_\P(g_{i+1})(z_{i})=z_i-\binom{Dg_{i+1}(z_i)}{z_i^*}^{-1}g_{i+1}(z_{i})$
\STATE $\tilde{z}_{i+1}\leftarrow$ any vector in $\Q[\ii]^{n+1}$ satisfying \label{step:tildeZ}
\begin{equation}\label{eq:errorzi}
  d_R(\tilde{z}_{i+1},z_{i+1})\leq \sqrt{\eps}=\frac{(1-\delta)u_0}{ 2d^{3/2} (1+3\delta u_0/2)\chi_{i,1} }.
\end{equation}
\STATE $z_{i+1}\leftarrow\tilde{z}_{i+1}$
\STATE $i \leftarrow i + 1$.
\ENDWHILE
\STATE $z_* \leftarrow \tilde{z}_{i+1}$.
\end{algorithmic}
\end{algorithmscheme}

Note that (\ref{eq:hyp}) implies that for $i\geq0$ we have
\begin{equation}\label{eq:hyp2}
-L_0\|g_i\|\,\|f-g_i\|\leq \Real{\langle g_i,f-g_i\rangle} \leq \|g_i\|\,\|f-g_i\|,\end{equation}
where $L_0=1-10^{-3}$.

\begin{theorem}\label{th:stepsizereallylinear}
The output of any algorithm performing the instructions described in \TrackReallyLinearHomotopy{} is certified. Namely, for every $i\geq0$, the point $\tilde{z}_i$ is an approximate zero of $g_i=f_{s_i}$, with associated zero \footnote{Note the slight abuse of notation: we just use $\zeta_i$ the zero of $g_i=f_{s_i}$, so we should actually denote it by $\zeta_{s_i}$. } $\zeta_i$, the unique zero of $g_i$ such that $(g_i,\zeta_i)$ lies in the lifted path $(f_t,\zeta_t)$. Moreover,
\[
d_R(\tilde{z}_i,\zeta_i)\leq\frac{u_0}{2d^{3/2}\mu(g_i,\zeta_i)},\;\;\;\;i\geq1.
\]

Let $\mathcal{C}_0$ be defined by (\ref{eq:condlength}) and (\ref{equ: linear homotopy}). If $\mathcal{C}_0<\infty$, there exists $k\geq0$ such that $f=g_{k}$. For the number of homotopy steps $k$ the following bounds hold:
\[
C'd^{3/2}\mathcal{C}_0\leq k\leq \lceil Cd^{3/2}\mathcal{C}_0\rceil,
\]
where
\[
C=\frac{c'RP}{c(1-\sqrt{2}u_0/2)^{1+\sqrt{2}}\ln(1+c')},
\qquad
C'=\frac{P}{c'}.
\]
In particular, if $\mathcal{C}_0<\infty$, there exists a unique lift $(f_t,\zeta_t)$ of the path $f_t$, and the algorithm finishes and outputs $z_*$, an approximate zero of
$f=g_k$ with associated zero $\zeta_k$, the unique zero  of $f$ such that $(f,\zeta_k)$ lies in the lifted path $(f_t,\zeta_t)$.

Finally, the two following inequalities hold at every step of the algorithm:
\begin{equation}\label{eq:varphivsmui}
\varphi_i\leq\frac{\sqrt{2}\mu(g_i,\zeta_i)^2}{(1-\sqrt{2}u_0/2)^{1+\sqrt{2}}},
\end{equation}
\begin{equation}\label{eq:UminusLBound}
U_R-L \geq\frac{\hat{c}}{d^3\max_{t\in[0,1]}\{\mu(f_t,\zeta_t)^4\}},\text{ where $\hat{c}$ is a universal constant.}
\end{equation}
\end{theorem}

\begin{remark}
As said above, we will choose $\delta=3/4$ in our main algorithm \TrackReallyLinearHomotopypractical{}. With that choice and the use of Frobenius norm instead of operator norm for the computation of $\chi_{i,1}$, in our practical implementation we will have
\[
28\leq C'\leq C\leq \constantForHomotopyComplexity \sqrt{n+1}.
\]
Note that $C$ is not a universal constant as it depends on $n$. 
The value of $\hat{c}$ is needed only for the bit--complexity analysis where it will be replaced by an $O(1)$. One can however estimate it as $\hat{c}\approx0.00003$.
\end{remark}

\begin{proof}
The proof of correctness of the algorithm is by induction on $i$. The base case of our induction $i=0$ follows. Assume that
\begin{equation}\label{eq:induction}
 d_R(z_i,\zeta_i)\leq\frac{u_0}{2d^{3/2}\mu(g_i,\zeta_i)}.
\end{equation}
We claim that we are under the hypotheses of Theorem \ref{th:onestep}. Indeed,
\[
g_{i+1}=f_{s_{i+1}}=f_{s_i+t_i}=\left(1-s_i-t_i\right)g+\left(s_i+t_i\right)f=
\]
\[
\left(1-s_i\right)g+s_if-t_ig+t_if=g_i+t_i(f-g).
\]
Thus,
\[
\frac{\Real{\langle g_i,g_{i+1}\rangle}}{\|g_i\|\|g_{i+1}\|}=\frac{\|g_i\|^2+t_i\Real{\langle g_i,f-g\rangle}}{\|g_i\|\sqrt{\|g_i\|^2+2t_i\Real{\langle g_i,f-g\rangle}+t_i^2\|f-g\|^2}}\geq L.
\]
Thus, from Lemma \ref{lem:LU} below we get
\begin{equation}\label{eq:pasoboundsalt1}
d_\S\left(\frac{g_i}{\|g_i\|},\frac{g_{i+1}}{\|g_{i+1}\|}\right)=\arccos\frac{\Real{\langle g_i,g_{i+1}\rangle}}{\|g_i\|\,\|g_{i+1}\|}\leq\arccos L<\frac{c}{Pd^{3/2}\varphi_i}.
\end{equation}
In particular, Theorem \ref{th:onestep} applies to the segment $[g_i,g_{i+1}]$, proving our induction step and also proving (\ref{eq:varphivsmui}) from (\ref{eq:varphivsmu}). Additionally, if the $i$--th step is not the final step in our algorithm (equivalently, $g_{i+1}\neq f$ or $s_{i+1}<1$) then we have
\[
\frac{\Real{\langle g_i,g_{i+1}\rangle}}{\|g_i\|\|g_{i+1}\|}=\frac{\|g_i\|^2+t_i\Real{\langle g_i,f-g\rangle}}{\|g_i\|\sqrt{\|g_i\|^2+2t_i\Real{\langle g_i,f-g\rangle}+t_i^2\|f-g\|^2}}\leq U_R,
\]
which again using Lemma \ref{lem:LU} yields:
\begin{equation}\label{eq:pasoboundsalt2}
d_\S\left(\frac{g_i}{\|g_i\|},\frac{g_{i+1}}{\|g_{i+1}\|}\right)=\arccos\frac{\Real{\langle g_i,g_{i+1}\rangle}}{\|g_i\|\,\|g_{i+1}\|}\geq  \arccos U_R>\frac{c}{RPd^{3/2}\varphi_i}.
\end{equation}

Now  we prove the bound on the number of steps. From (\ref{eq:teorconc2}) and (\ref{eq:pasoboundsalt2}) we have that, as long as $s_{i+1}< 1$,
\[
\int_{s_i}^{s_{i+1}=s_i+t_i}\mu(f_t,\zeta_t)\sqrt{\frac{\|\dot f_t\|^2}{\|f_t\|^2}-\frac{\Real{(\langle \dot{f},f_t\rangle)}^2}{\|f_t\|^4}+\|\dot\zeta_t\|_{T_{\zeta_t}\Pn}^2}\,dt\geq
\]
\[
\varphi_i d_\S\left(\frac{g_i}{\|g_i\|},\frac{g_{i+1}}{\|g_{i+1}\|}\right) \frac{(1-\sqrt{2}u_0/2)^{1+\sqrt{2}}\ln(1+c')}{c'}>
\]
\[
 \frac{c(1-\sqrt{2}u_0/2)^{1+\sqrt{2}}\ln(1+c')}{c'RPd^{3/2}}.
\]
Thus, as long as $s_{i+1}<1$, we have
\[
\mathcal{C}_0=\int_0^1\mu(f_t,\zeta_t)\sqrt{\frac{\|\dot f_t\|^2}{\|f_t\|^2}-\frac{\Real{(\langle \dot{f},f_t\rangle)}^2}{\|f_t\|^4}+\|\dot\zeta_t\|_{T_{\zeta_t}\Pn}^2}\,dt\geq
\]
\[
\int_0^{s_{i+1}}\mu(f_t,\zeta_t)\sqrt{\frac{\|\dot f_t\|^2}{\|f_t\|^2}-\frac{\Real{(\langle \dot{f},f_t\rangle)}^2}{\|f_t\|^4}+\|\dot\zeta_t\|_{T_{\zeta_t}\Pn}^2}\,dt=
\]
\[
\sum_{j=0}^i\int_{s_j}^{s_{j+1}}\mu(f_t,\zeta_t)\sqrt{\frac{\|\dot f_t\|^2}{\|f_t\|^2}-\frac{\Real{(\langle \dot{f},f_t\rangle)}^2}{\|f_t\|^4}+\|\dot\zeta_t\|_{T_{\zeta_t}\Pn}^2}\,dt>
\]
\[
\frac{(i+1)c(1-\sqrt{2}u_0/2)^{1+\sqrt{2}}\ln(1+c')}{c'RPd^{3/2}}.
\]
In particular, if $s_{i+1}<1$ then
\[
i+1<\frac{c'RPd^{3/2}}{c(1-\sqrt{2}u_0/2)^{1+\sqrt{2}}\ln(1+c')}\mathcal{C}_0.
\]
The first non--negative integer $i$ which violates this inequality is thus an upper bound for the number of iterations of the algorithm. This finishes the proof of the upper bound on the number of steps. For the lower bound, note that, even if $s_{i+1}=1$, from (\ref{eq:teorconc2}) and (\ref{eq:pasoboundsalt1}) we have
\[
\int_{s_i}^{s_{i+1}}\mu(f_t,\zeta_t)\sqrt{\frac{\|\dot f_t\|^2}{\|f_t\|^2}-\frac{\Real{(\langle \dot{f},f_t\rangle)}^2}{\|f_t\|^4}+\|\dot\zeta_t\|_{T_{\zeta_t}\Pn}^2}\,dt\,\leq
\]
\[
\varphi_i d_\S\left(\frac{g_i}{\|g_i\|},\frac{g_{i+1}}{\|g_{i+1}\|}\right) \frac{1+\sqrt{2}u_0/2}{(1-\sqrt{2}u_0/2)^{\sqrt{2}}}<  \frac{c(1+\sqrt{2}u_0/2)}{Pd^{3/2}(1-\sqrt{2}u_0/2)^{\sqrt{2}}}\leq \frac{c'}{Pd^{3/2}}.
\]
Thus, if $k$ is the number of iterations needed by the algorithm (i.e. $s_{k-1}<s_{k}=1$) then
\[
\mathcal{C}_0=\int_0^1\mu(f_t,\zeta_t)\sqrt{\frac{\|\dot f_t\|^2}{\|f_t\|^2}-\frac{\Real{(\langle \dot{f},f_t\rangle)}^2}{\|f_t\|^4}+\|\dot\zeta_t\|_{T_{\zeta_t}\Pn}^2}\,dt=
\]
\[
\sum_{j=0}^{k-1}\int_{s_{j}}^{s_{j+1}}\mu(f_t,\zeta_t)\sqrt{\frac{\|\dot f_t\|^2}{\|f_t\|^2}-\frac{\Real{(\langle \dot{f},f_t\rangle)}^2}{\|f_t\|^4}+\|\dot\zeta_t\|_{T_{\zeta_t}\Pn}^2}\,dt<
k \frac{c'}{Pd^{3/2}}.
\]
In particular, we conclude that the total number of iterations is
\[
k\geq \frac{Pd^{3/2}\mathcal{C}_0}{{c'}},
\]
which is the lower bound on $k$ claimed in the theorem.

For (\ref{eq:UminusLBound}), note that
\[
U_R-L =\frac{c^2}{2P^2d^3\varphi_i^2}\left(1-\frac{1}{R^2}-\frac{c^2}{12P^2d^3\varphi_i^2}\right).
\]
Using $R\geq\sqrt{2}$ and roughly bounding the term inside the parenthesis, we get
\[
U_R-L \geq\frac{c^2}{5P^2d^3\varphi_i^2}\underset{(\ref{eq:varphivsmui})}{\geq}\frac{c^2\left((1-\sqrt{2}u_0/2)^{1+\sqrt{2}}\right)^2}{10P^2d^3\mu(g_i,\zeta_i)^4},
\]
which implies (\ref{eq:UminusLBound}).

\end{proof}
\begin{lemma}\label{lem:LU}
Let $L,U_R$ be defined as in the algorithm. Then,
\[
\arccos U_R>\frac{c}{RPd^{3/2}\varphi_i},\qquad\arccos L<\frac{c}{Pd^{3/2}\varphi_i}.
\]
\end{lemma}
\begin{proof}
We prove the second inequality. Recall the elementary fact that for $0<s<1$ we have
\[
\cos(s)<1-\frac{s^2}{2}+\frac{s^4}{24}.
\]
Then, because $\arccos$ is a decreasing function in $[0,1]$,
\begin{equation}\label{eq:arccosbound}
\arccos\left(1-\frac{s^2}{2}+\frac{s^4}{24}\right)<s,\qquad s\in(0,1).
\end{equation}
In particular,
\[
\arccos L=\arccos \left(1-\frac{c^2}{2P^2d^3\varphi_i^2}+\frac{c^4}{24P^4d^6\varphi_i^4}\right)<\frac{c}{Pd^{3/2}\varphi_i},
\]
as desired. The first inequality is proved in the same way, using that for $0<s<1$ we have
\[
\cos(s)>1-\frac{s^2}{2}.
\]
\end{proof}

\section{Computational considerations}

Provided Theorem~\ref{th:stepsizereallylinear}, the rigorously certified homotopy tracking could be accomplished by way of exact rational arithmetic employed in all of the computations described in \TrackReallyLinearHomotopy{}. In this section we discuss some of the aspects of this issue, to facilitate the reading of our main algorithm \TrackReallyLinearHomotopypractical{} below.

\subsection{Operator norm vs. Frobenius norm}\label{sec:operatorvsfrobenius}
In Step 4 of \TrackReallyLinearHomotopy{} we need to compute the operator norm of a matrix, which is a non--trivial task. Actually, one just needs to the square of such norm, to use it in steps 5 and 10. Instead of computing the square of the operator norm, one can just compute the square of the Frobenius norm $\|(a_{ij})\|_F^2=\sum_{i,j}|a_{ij}|^2$, which involves only rational computations. Both norms are related by the inequalities
\[
\|\cdot\|^2\leq\|\cdot\|_F^2\leq (n+1)\|\cdot\|^2.
\]
On the other hand, $\chi_{i,2}^2$ which is the squared norm of a vector involves only rational computations and thus can be computed exactly. Then, instead of $\varphi_i^2$ in Step 5 we can use the product of $\chi_{i,2}^2$ and a version of $\chi_{i,1}^2$ using the squared Frobenius norm.

Let us put this in a general framework. Assume that we can compute some quantity $\tilde{\chi}_{i,1}^2$ satisfying $\chi_{i,1}^2\leq\tilde{\chi}_{i,1}^2\leq S^2\chi_{i,1}^2$ for some $S\geq1$. Let $\tilde{L},\tilde{U}_{\sqrt{2}}$ be computed with the same formulas as $L,U_{\sqrt{2}}$ but using $\tilde{\chi}_{i,1}$ instead of $\chi_{i,1}$. Then, it is easy to see that $\tilde{L}\geq L$ and
\[
\tilde{U}_{\sqrt{2}}=1-\frac{c^2}{4P^2d^3\tilde{\varphi}_i^2}\leq 1-\frac{c^2}{4S^2P^2d^3\varphi_i^2}=U_{\sqrt{2}S}.
\]
Thus, if we find $t_i$ such that
\[
\tilde{L} \leq \frac{\|g_i\|^2+t_i\Real{\langle g_i,f-g\rangle}}{\|g_i\|\sqrt{\|g_i\|^2+2t_i\Real{\langle g_i,f-g\rangle}+t_i^2\|f-g\|^2}}\leq \tilde{U}_{\sqrt{2}}\leq U_{\sqrt{2}S},
\]
then in particular the hypotheses of Theorem \ref{th:stepsizereallylinear} are fulfilled changing $R$ to $\sqrt{2}S$ and the number of steps is at most
\[
k\leq \lceil \frac{\sqrt{2} SPc'}{c(1-\sqrt{2}u_0/2)^{1+\sqrt{2}}\ln(1+c')}d^{3/2}\mathcal{C}_0\rceil.
\]
In particular, we have proved the following.

\begin{lemma}\label{lem:frobenius}
If in \TrackReallyLinearHomotopy{}, $R$ is changed to $\sqrt{2}$ and $\chi_{i,1}^2$ is changed to $\tilde{\chi}_{i,1}^2$ (defined the same way as $\chi_{i,1}^2$ but using Frobenius norm instead of the operator norm), then any algorithm performing the computations in \TrackReallyLinearHomotopy{} has certified output in the sense of Theorem \ref{th:stepsizereallylinear}. The number of iterations is at most
\[
k\leq \lceil \frac{\sqrt{2(n+1)}Pc'}{c(1-\sqrt{2}u_0/2)^{1+\sqrt{2}}\ln(1+c')}d^{3/2}\mathcal{C}_0\rceil\underset{\text{with }\delta=3/4}{\approx} \lceil  \constantForHomotopyComplexity\sqrt{n+1}d^{3/2}\mathcal{C}_0\rceil.
\]
Moreover, at every step we have
\begin{equation}\label{eq:varphivsmuibis}
\tilde{\chi}_{i,1}\chi_{i,2}\leq\frac{\sqrt{2(n+1)}\mu(g_i,\zeta_i)^2}{(1-\sqrt{2}u_0/2)^{1+\sqrt{2}}},
\end{equation}
\begin{equation}\label{eq:UminusLBoundbis}
U_{\sqrt{2(n+1)}}-L \geq\frac{\hat{c}}{nd^3\max\{\mu(f_t,\zeta_t)^4\}},
\end{equation}
where $\hat{c}$ is a universal constant.
\end{lemma}

\subsection{Computing the step size}
Step 5 of \TrackReallyLinearHomotopy{} requires finding a $t$ satisfying (\ref{eq:newstepeqn}), which {\em a priori} means computing approximately the smallest positive roots of two quadratic polynomials. An easier way to get this is using bisection method to locate a root of the equation
\[
\alpha(t)=\frac{\theta_1+t\theta_2}{\sqrt{\theta_1(\theta_1+2t\theta_2+t^2\theta_3)}}-\frac{L+U}{2}=0,
\]
with stopping criterion given by
\[
\left|\alpha(t)\right|\leq\frac{U-L}{2}.
\]
If $t\in\R$ satisfies this stopping criterion, then (taking $\theta_1=\|g_i\|^2,\theta_2=\Real{\langle g_i,f-g\rangle},\theta_3=\|f-g\|^2$) it also satisfies (\ref{eq:newstepeqn}). When applying bisection, we need to be able to determine the sign of $\alpha(t)$. It is not hard to accomplish this without computing square roots using the following subroutine.

\begin{algorithm}\label{alg:sign}$(s,r)=\mbox{\Signcompute{}}(\theta_1,\theta_2,\theta_3,t,L,U)$
\begin{algorithmic}[1]
\REQUIRE $\theta_1,\theta_2,\theta_3,t,L,U\in\Q$, $\theta_1,\theta_3>0$, $\theta_2^2<\theta_1\theta_3$, $0<L<U<1$.
\ENSURE $s=1$ if $\alpha(t)>0$, $s=-1$ otherwise, and
\[
r=\frac{(\theta_1+t\theta_2)^2}{\theta_1(\theta_1+2t\theta_2+t^2\theta_3)}
\]
\STATE $\displaystyle r\leftarrow \frac{(\theta_1+t\theta_2)^2}{\theta_1(\theta_1+2t\theta_2+t^2\theta_3)}$
\IF{$\theta_1+t\theta_2>0$ and $r>(L+U)^2/4$}
\STATE $s\leftarrow 1$
\ELSE
\STATE $s\leftarrow -1$.
\ENDIF
\end{algorithmic}
\end{algorithm}

The bisection method mentioned above is then as follows. The requirement $1-10^{-3}<L$ in the description of the following algorithm corresponds to the extra hypotheses (\ref{eq:hyp}) in our main algorithm.

\begin{algorithm}\label{alg:compute_approx_quadratic}$t=\mbox{\LowerUpperBoundt{}}(\theta_1,\theta_2,\theta_3,L,U)$
\begin{algorithmic}[1]
\REQUIRE $\theta_1,\theta_2,\theta_3,L,U\in\Q$, $\theta_1,\theta_3>0$, $-L\sqrt{\theta_1\theta_3}\leq \theta_2< \sqrt{\theta_1\theta_3}$, $1-10^{-3}<L<U<1$.
\ENSURE $t=\frac{m}{2^l}\in\Q\cap(0,1]$, $m\in\Z$, $0\leq l\in \Z$ such that
\begin{equation}\label{eq:formLU}
L\leq \frac{\theta_1+t\theta_2}{\sqrt{\theta_1(\theta_1+2t\theta_2+t^2\theta_3)}}\leq U,
\end{equation}
if such $t$ exists (otherwise, output $t=1$), and such that
\begin{equation}\label{eq:boundml}
0<m\leq 2^l\leq \max\left(1,\frac{16\theta_3}{\theta_1(U-L)}\right).
\end{equation}
\STATE $t_1\leftarrow1$
\STATE $L_2\leftarrow L^2$
\STATE $(s_1,r_1)\leftarrow\mbox{\Signcompute{}}(\theta_1,\theta_2,\theta_3,t_1,L,U)$
\IF{$\theta_1+\theta_2>0$ and $r_1\geq L_2$}\label{step:beta1geqL}
\STATE $t\leftarrow t_1$
\ELSE
\STATE $U_2\leftarrow U^2$
\STATE $t_0\leftarrow0$
\STATE $\displaystyle t_2\leftarrow\frac{t_0+t_1}{2}$
\STATE $(s_2,r_2)\leftarrow\mbox{\Signcompute{}}(\theta_1,\theta_2,\theta_3,t_2,L,U)$
\STATE $l\leftarrow0$
\WHILE{$L_2>r_2$ or $U_2<r_2$ or $\theta_1+t_2\theta_2<0$}\label{line:while}
\IF{$s_2=1$}
\STATE $t_0\leftarrow t_2$
\ELSE
\STATE $t_1\leftarrow t_2$
\ENDIF
\STATE $\displaystyle t_2\leftarrow\frac{t_0+t_1}{2}$
\STATE $(s_2,r_2)\leftarrow\mbox{\Signcompute{}}(\theta_1,\theta_2,\theta_3,t_2,L,U)$
\STATE $l\leftarrow l+1$
\ENDWHILE
\STATE $t\leftarrow t_2$
\ENDIF

\end{algorithmic}
\end{algorithm}
\begin{lemma}\label{lem:bisection}
\LowerUpperBoundt{} produces $t=m/2^l$ satisfying (\ref{eq:formLU}) and (\ref{eq:boundml}), or $t=1$ in case there exists no $t$ that satisfies (\ref{eq:formLU}). Moreover, the number of iterations it performs is at most
\[
O\left(\log_2\frac{\theta_3}{\theta_1(U-L)}\right).
\]
\end{lemma}
\begin{proof}
Let
\[
\beta(t)=\frac{\theta_1+t\theta_2}{\sqrt{\theta_1(\theta_1+2t\theta_2+t^2\theta_3)}}.
\]
We first claim that
\[
\beta'(t)=\alpha'(t)=-\frac{\left( \theta_1\,\theta_3-{\theta_2}^{2}\right) \,t}{\sqrt{\theta_1}\,{\left( \theta_3\,{t}^{2}+2\,\theta_2\,t+\theta_1\right) }^{\frac{3}{2}}}
\]
This is a routine computation and is left to the reader. In particular, $\theta_2^2<\theta_1\theta_3$ implies that $\alpha(t)$ and $\beta(t)$ are decreasing functions for $t\geq0$. If $\beta(1)\geq L$ (which is decided in Step~\ref{step:beta1geqL} of Algorithm~\ref{alg:compute_approx_quadratic}) then there are two possible scenarios:
\begin{enumerate}
\item If $\beta(1)>U$ then a $t$ satisfying (\ref{eq:formLU}) does not exist and the output of the algorithm is $t=1$ as claimed.
\item If $\beta(1)\leq U$ then the output of the algorithm $t=1$ satisfies (\ref{eq:formLU}) and (\ref{eq:boundml}) as claimed.
\end{enumerate}
On the other hand, if $\beta(1)\leq L$, then
\[
\alpha(0)=1-\frac{L+U}{2}>0,\qquad \alpha(1)=\beta(1)-\frac{L+U}{2}\leq \frac{L-U}{2}<0,
\]
which implies that the bisection method used in the algorithm produces an approximation of the unique root $t_*\in(0,1)$ of $\alpha(t)$. In particular, note that $\alpha(t_*)=0$ implies that $\theta_1+t_*\theta_2>0$ and
\[
\beta(t_*)=\frac{L+U}{2}\in(L,U),
\]
which yields $\beta(t_*)^2\in(L^2,U^2)$. By continuity of $\beta$, we conclude that the algorithm will at some point compute a $t_2$ such that $L^2\leq r_2\leq U^2$ and $\theta_1+t_2\theta_2>0$. That is, the algorithm finishes at some point, and the output satisfies (\ref{eq:formLU}) as claimed. It is a simple induction exercise to prove that, whenever the condition of Line~\ref{line:while} is satisfied (that is to say, at every step of the algorithm, except possibly at the last one) we have
\[
[p,q]\subseteq[t_0,t_1],
\]
where
\begin{align*}
[p,q]&=\{t\in [0,1]:|\alpha(t)|\leq (U-L)/2\}\\&=\{t\in[0,1]:\beta(t)\in[L,U]\}=\beta^{-1}([L,U]).
\end{align*}
At every step of the algorithm, the bisection method satisfies
\[
|t_1-t_0|=\frac{1}{2^l}.
\]
Thus, if the $l$--th step is not the last step of the algorithm then we have
\[
q-p\leq \frac{1}{2^l}.
\]
From the Mean Value Theorem of calculus, we have
\[
\frac{U-L}{q-p}=\frac{|\beta(q)-\beta(p)|}{|q-p|}=\beta'(\hat{t}\,),
\]
for some $\hat{t}\in[p,q]=\beta^{-1}([L,U])$.
Thus,
\begin{equation}\label{eq:auxn}
2^l\leq \frac{1}{q-p}=\frac{\beta'(\hat t\,)}{U-L}\leq\frac{\max\{|\beta'(t)|:\beta(t)\in[L,U]\}}{U-L}.
\end{equation}
Note moreover that
\begin{equation}\label{eq:betapbeta}
\frac{|\beta'|}{\beta^3}=\frac{(\theta_1\theta_3-\theta_2^2)\theta_1t}{|\theta_1+\theta_2t|^3}.
\end{equation}
Now we have to distinguish two cases:
\begin{enumerate}
\item If $\theta_2\geq0$ then (\ref{eq:betapbeta}) and $t\leq 1$ yield
\[
|\beta'(t)|\leq\frac{(\theta_1\theta_3-\theta_2^2)\theta_1}{\theta_1^3}\leq\frac{\theta_3}{\theta_1}.
\]
\item If $\theta_2< 0$ then by hypotheses we have $\theta_2=-e$ with $0<e\leq L\sqrt{\theta_1\theta_3}$. Thus,
\[
\beta\left(\frac{\theta_1}{2e}\right)=\frac{\theta_1/2}{\sqrt{\theta_1}\sqrt{\frac{\theta_3\theta_1^2}{4e^2}}}=\frac{e}{\sqrt{\theta_1\theta_3}}\leq L,
\]
and hence $\beta(t)\geq L$ implies that $t\leq \theta_1/(2e)$. Thus, (\ref{eq:betapbeta}) implies that
\[
\frac{|\beta'|}{\beta(t)^3}\leq\frac{(\theta_1\theta_3-e^2)\theta_1t}{(\theta_1/2)^3}\underset{t\leq1}\leq\frac{8(\theta_1\theta_3-e^2)}{\theta_1^2}\leq\frac{8\theta_3}{\theta_1},\qquad t\in[p,q],
\]
which readily gives
\[
|\beta'(t)|\leq\beta^3(t)\frac{8\theta_3}{\theta_1}\leq\frac{8\theta_3}{\theta_1},\qquad t\in[p,q].
\]
\end{enumerate}
Thus, for every possible value of $\theta_2\in(-L\sqrt{\theta_1\theta_3},\sqrt{\theta_1\theta_3})$ we have that
\[
\max\{|\beta'(t)|:\beta(t)\in[L,U]\}\leq\frac{8\theta_3}{\theta_1}.
\]
This together with (\ref{eq:auxn}) proves that, if the $l$-th step is not the last step of the algorithm, we have
\[
2^l\leq \frac{8\theta_3}{\theta_1(U-L)}.
\]
For the last step, this quantity has to be multiplied by $2$. The last claim of the lemma follows.
\end{proof}

\begin{remark}
  Our implementation of \LowerUpperBoundt{} continues bisection if the denominator of its output $t$ is larger than the denominator of $s_i$ in step \ref{line: where LUquadratic is called} of \TrackReallyLinearHomotopypractical{} (Algorithm~\ref{alg:practical_homotopy}) until the denominators match. This is done in order to reduce the size of the denominator of $s_{i+1}$.
\end{remark}

\subsection{Finding a close-by number with small integer coordinates}
In Step~\ref{step:tildeZ} of \TrackReallyLinearHomotopy{} we change $z_{i+1}$ to a close-by vector $\tilde{z}_{i+1}$ with rational coordinates. Although $z_{i+1}$ already has rational coordinates, we need to replace $z_{i+1}$ with a nearby vector whose coordinates are integer numbers of bounded (small) absolute value. If this step is not performed, the number of bits required to write up $z_{i+1}$ might increase at each loop, which is to be avoided. In this section we show how to deal with the general problem of, given $z\in\Q[\ii]^{n+1}$ and $\eps\in\Q$, $\eps>0$, finding $\tilde{z}\in\Z[\ii]^{n+1}$ such that\footnote{Recall that $d_R(x,y)=\arccos\frac{\langle x,y\rangle}{\|x\|\|y\|}$ is the usual distance from $x$ to $y$ as projective points in $\Pn$.} $d_R(\tilde{z},z)\leq\sqrt{\eps}$, and such a way that the absolute value of the coordinates of $\tilde{z}$ is relatively small.

Let us consider the following algorithm.
\begin{algorithm} \label{alg:compute_approx_newton}$\tilde{z}=\mbox{\Computeapproxzero{}}(z,\eps)$
\begin{algorithmic}[1]
\REQUIRE $z\in\Q[\ii]^{n+1}$; $\eps\in\left(0,\frac{1}{5}\right)\cap\Q$.
\ENSURE $\tilde{z}\in\Z[i]^{n+1}$ such that
\begin{equation}\label{eq:errorzialg}
  d_R(\tilde{z},z)\leq \sqrt{\eps},
\end{equation}
and such that the integer numbers appearing in the expression of $\tilde{z}$ are bounded in absolute value by $3\sqrt{\frac{n+1}{\eps}}$.

\STATE Let $\frac{a_i}{c_i}+\ii\frac{b_i}{e_i}$, $a_i,b_i,c_i,e_i\in\Z$, $c_i,e_i>0$, $0\leq i\leq n$ be the coordinates of $z$.
\STATE $m\leftarrow (c_0\cdots c_n)\cdot (e_0\cdots e_n)$.
\STATE $x\leftarrow m\cdot z$.
\STATE $\displaystyle r\leftarrow \left(\frac{21}{20}\right)^2$
\STATE $k\leftarrow 0$
\STATE $\alpha\leftarrow 4$
\WHILE{$\alpha\leq\frac{\eps\|x\|^2}{2(n+1)r}$}
\STATE $\alpha\leftarrow 4\alpha$
\STATE $k\leftarrow k+1$
\ENDWHILE
\STATE $\tilde{z}\leftarrow [ 2^{-k} x]$.
\end{algorithmic}
\end{algorithm}
Here, by $[y]$ ($y\in\C^{n+1}$) we mean the following: if $y=(a_0+\ii b_0,\ldots,a_n+\ii b_n)$ then
\[
[ y]=([ a_0]+\ii [ b_0],\ldots,[ a_n]+\ii [ b_n]),
\]
where for $t\in\R$, $[t]$ is the integer number which is closest to $t$ and is smaller than $t$ in absolute value (that is, $[t]=\lfloor t\rfloor$ if $t\geq0$ and $[t]=\lceil t\rceil$ if $t<0$).

\begin{lemma}\label{lem:minim}
Let $0\leq \theta_3<\theta_1$. Then, the function
\[
w(\theta_2)=\frac{\theta_1+\theta_2}{\sqrt{\theta_1(\theta_1+ 2\theta_2 +\theta_3)}},\qquad \theta_2\in[-\sqrt{\theta_1\theta_3},\sqrt{\theta_1\theta_3}]
\]
has a global minimum value equal to $\sqrt{1-\theta_3/ \theta_1}$.
\end{lemma}
\begin{proof}
Note that $w$ is a differentiable function and
\[
w'(\theta_2)=\frac{\theta_3+\theta_2}{\sqrt{\theta_1}(\theta_1+2\theta_2+\theta_3)^{3/2}}.
\]
Hence, the minimum of $w$ is attained at $\theta_2=-\sqrt{\theta_1\theta_3}$, $\theta_2=\sqrt{\theta_1\theta_3}$ or $\theta_2=-\theta_3$. Now, $w(\sqrt{\theta_1\theta_3})=w(-\sqrt{\theta_1\theta_3})=1$ and $w(-\theta_3)=\sqrt{1-\theta_3/\theta_1}\leq1$. The lemma follows.
\end{proof}

\begin{lemma}\label{lem:findapproxzero}
Algorithm \ref{alg:compute_approx_newton} produces $\tilde z=(\tilde{\alpha}_0+\ii \tilde\beta_0,\ldots,\tilde\alpha_{n}+\ii \tilde\beta_n)\in\Z[\ii]^{n+1}$ satisfying (\ref{eq:errorzialg}) and such that
\[
|\tilde\alpha_i|,|\tilde\beta_i|\leq 3\sqrt{\frac{n+1}{\eps}}\qquad \forall\; 0\leq i\leq n.
\]
\end{lemma}
\begin{proof}
First note that, if the stopping condition of the loop is satisfied at the first step, then the output $\tilde{z}=x$ of the algorithm satisfies $d_R(\tilde{z},z)=0$ and
\[
\|x\|\leq\sqrt{\frac{8(n+1)r}{\eps}}\leq3\sqrt{\frac{n+1}{\eps}},
\]
and hence the claim of the lemma follows. Otherwise, the numbers $\alpha,k$ computed by the algorithm satisfy
\begin{equation}\label{eq:k}
\alpha=4^{k+1},\qquad 4^k\leq \frac{\eps\|x\|^2}{2(n+1)r}<4^{k+1}.
\end{equation}
Let $x=(\alpha_0+\ii\beta_0,\ldots,\alpha_n+\ii \beta_n)$ be the coordinates of $x$. Then, for $i=0,\ldots,n$, we have:
\[
|(\tilde\alpha_i+\ii\tilde\beta_i)-2^{-k}(\alpha_i+\ii\beta_i)|^2=([2^{-k} \alpha_i]-2^{-k}\alpha_i)^2+([ 2^{-k}\beta_i]-2^{-k}\beta_i)^2<2.
\]
Hence, denoting $y=2^{-k}x$ and $v=\tilde{z}-y$ we have
\begin{equation}\label{eq:boundva}
\|v\|^2=\sum_{i=0}^n |(\tilde\alpha_i+\ii\tilde\beta_i)-2^{-k}(\alpha_i+\ii\beta_i)|^2\leq 2(n+1).
\end{equation}
On the other hand,
\begin{align*}
\|y\|^2-\|v\|^2&=\|y\|^2-\|\tilde{z}-y\|^2=2\Real{\langle\tilde{z},y\rangle}-\|\tilde{z}\|^2\\
&=2\left(\sum_{i=0}^n2^{-k}\alpha_i[2^{-k}\alpha_i]+2^{-k}\beta_i[2^{-k}\beta_i]\right)-\|\tilde{z}\|^2\\
&\geq2\left(\sum_{i=0}^n[2^{-k}\alpha_i]^2+[2^{-k}\beta_i]^2\right)-\|\tilde{z}\|^2\\
&=2\|\tilde{z}\|^2-\|\tilde{z}\|^2\geq0.
\end{align*}
That is, $\|v\|^2\leq\|y\|^2$. Hence, the use of Lemma \ref{lem:minim} in the following chain of inequalities is justified:
\[
\frac{|\langle y+v,y\rangle|}{\|y\|\|y+v\|}\geq\frac{\Real{\langle y+v,y\rangle}}{\|y\|\|y+v\|}=\frac{\|y\|^2+\Real{\langle v,y\rangle}}{\|y\|\|y+v\|}=
\]
\[
 \frac{\|y\|^2+\Real{\langle v,y\rangle}}{\|y\|\sqrt{\|y\|^2+2\Real{\langle v,y\rangle}+\|v\|^2}}\underset{Lemma\; \ref{lem:minim}}{\geq}\sqrt{1-\frac{\|v\|^2}{\|y\|^2}}.
\]
Thus,
\[
d_R(\tilde{z},z)=d_R(\tilde{z},x)=d_R(\tilde{z},y)=d_R(y+v,y)=\arccos\frac{|\langle y+v,y\rangle|}{\|y\|\|y+v\|}\leq
\]
\[
\arccos \sqrt{1-\frac{\|v\|^2}{\|y\|^2}}=\arcsin\frac{\|v\|}{\|y\|}=\arcsin\frac{2^k\|v\|}{\|x\|}\underset{(\ref{eq:boundva})}{\leq} \arcsin\frac{2^k\sqrt{2(n+1)}}{\|x\|}.
\]
Note from (\ref{eq:k}) that
\begin{equation}\label{eq:boundv}
\frac{2^k\sqrt{2(n+1)}}{\|x\|}\leq \sqrt{\frac{\eps}{r}}\leq\sqrt{\eps}\leq\frac{1}{2}.
\end{equation}
The reader can check that the function $s\mapsto s^{-1}\arcsin s$, $s\in[0,1)$ is an increasing function. From this fact and (\ref{eq:boundv}) we get:
\[
\frac{\arcsin\frac{2^k\sqrt{2(n+1)}}{\|x\|}}{\frac{2^k\sqrt{2(n+1)}}{\|x\|}}\leq \frac{\arcsin \frac{1}{2}}{\frac{1}{2}}\leq \frac{21}{20}=\sqrt{r},
\]
which readily implies
\[
d_R(\tilde{z},z)\leq \arcsin\frac{2^k\sqrt{2(n+1)}}{\|x\|}\leq \sqrt{r}\;\frac{2^k\sqrt{2(n+1)}}{\|x\|}\underset{(\ref{eq:k})}{\leq} \sqrt{\eps},
\]
as wanted.

For the bound on $|\tilde{\alpha}_i|$ note that
\[
|\tilde{\alpha}_i|=\left|\left[2^{-k}\alpha_i\right]\right|\leq \left|2^{-k}\alpha_i\right|\underset{(\ref{eq:k})}{\leq}\frac{\sqrt{8(n+1)r}}{\|x\|\sqrt{\eps}}|\alpha_i|\leq 3\sqrt{\frac{n+1}{\eps}},
\]
where we have used $\|x\|\geq|\alpha_i|$. An identical chain of inequalities works for $\tilde{\beta}_i$.

\end{proof}

\section{The main algorithm}
We now describe the pseudo-code of an actual algorithm that performs the instructions described in \TrackReallyLinearHomotopy{} and is thus certified.

There are two choices in \TrackReallyLinearHomotopy{}: $R$ and $\delta$. We choose $R=\sqrt{2}$ and $\delta=3/4$ which make the computations simple. Besides, instead of using operator norm for the computation of $\chi_{i,1}^2$ we use Frobenius norm, which according to Section \ref{sec:operatorvsfrobenius} multiplies by a factor of $\sqrt{n+1}$ the upper bound for the number of homotopy steps. The reader may find helpful Table \ref{table:notationcompee} for comparing the names of the variables in \TrackReallyLinearHomotopy{} and \TrackReallyLinearHomotopypractical{}:

\begin{table}
\caption{Notation of \TrackReallyLinearHomotopypractical{} and \TrackReallyLinearHomotopy{}.}
\label{table:notationcompee}
\begin{tabular}{|c|c|}
\hline
\TrackReallyLinearHomotopypractical{}& \TrackReallyLinearHomotopy{}\\
\hline
$n_1$ & $\|f\|^2$\\
\hline
$n_2$ & $\|g\|^2$\\
\hline
$n_3$ & $\Real{\langle f,g\rangle}$\\
\hline
$\dot n$ & $\|f-g\|^2$\\
\hline
$n_4$ & $\|g_i\|^2$\\
\hline
$n_5$ & $\Real{\langle f,g_i\rangle}$\\
\hline
$n_6$ & $\Real{\langle f-g,g_i\rangle}$\\
\hline
$n_7$ & $\|z_i\|^2$\\
\hline
$v_1$ & $f(z_i)$ \\
\hline
$v_2$ & $g_i(z_i)$\\
\hline
$M$ & $\binom{Dg_i({z_i})}{{z_i}^*}^{-1}$\\
\hline
$\tilde{M}$ & $\binom{Dg_{i+1}({z_i})}{{z_i}^*}$\\
\hline
$\mfa$ & $\tilde\chi_1^2$ ($=\chi_1^2$ computed with Frobenius norm)\\
\hline
$\mfb$ & $\chi_2^2$\\
\hline
$\mfa\mfb$ & $\tilde\varphi^2=\tilde\chi_1^2\chi_2^2$ (plays the role of $\varphi^2=\chi_1^2\chi_2^2$)\\
\hline
$W$ & $\frac{c^2}{2P^2d^3\tilde\varphi^2}$\\
\hline
\end{tabular}
\end{table}

\begin{algorithm} \label{alg:practical_homotopy}$z_*=\mbox{\TrackReallyLinearHomotopypractical{}}(f,g,z_0)$
\begin{algorithmic}[1]
\REQUIRE $f,g\in\CHd$; $z_0\in\Q[i]^{n+1}$ is an approximate zero of $g$ satisfying (\ref{eq:hypz0}).
\ENSURE $z_*\in\Z[i]^{n+1}$ is an approximate zero of $f$ associated to the end of the homotopy path starting at the zero of $g$
associated to $z_0$
and defined by the homotopy (\ref{equ: linear homotopy}).
\STATE $i \leftarrow 0$; $s_i=0$.
\STATE $n_1\leftarrow\|f\|^2$.
\STATE $n_2\leftarrow\|g\|^2$.
\STATE $n_3\leftarrow\Real{\langle f,g\rangle}$.
\STATE $\dot n\leftarrow n_1^2+n_2^2-2n_3$.
\STATE $\displaystyle \eps_0\leftarrow \frac{u_0^2}{ (4d)^{3}(1+9 u_0/8)^2 }$
\STATE $\displaystyle W_0\leftarrow \ConstantctwooverPtwo\frac{1}{d^3}$
\WHILE {$s_i <1$}
\STATE $n_4\leftarrow (1-s_i)^2n_2+s_i^2 n_1+2s_i(1-s_i)n_3$
\STATE $n_5\leftarrow (1-s_i)n_3+s_i n_1$
\STATE $n_6\leftarrow s_i n_1-(1-s_i)n_2+(1-2s_i)n_3$
\STATE $n_7\leftarrow \|z_i\|^2$
\STATE $M_1\leftarrow Dg(z_i)$; $M_2\leftarrow Df(z_i)$.
\STATE
\[
M=(m_{ij})\leftarrow\binom{(1-s_i)M_1+s_iM_2}{{z_i}^*}^{-1}\in\mathcal{M}_{n+1}(\C).
\]
\STATE
\[
\mfa\leftarrow \left(n_4\sum_{k=0}^{n}\sum_{l=0}^{n-1}d_{l+1}|m_{kl}|^2n_7^{d_{l+1}-1}\right)+\left(\sum_{k=0}^n|m_{kn}|^2n_7 \right)
\]
\STATE $v_1\leftarrow f(z_i)\in\C^n$
\STATE $v_2\leftarrow g_i(z_i) = (1-s_i)g(z_i) + s_i v_1\in \C^n$
\STATE $v_3\leftarrow n_4 v_1 - n_5 v_2$ 
\STATE $v_4\leftarrow M\binom{v_3}{0}\in\C^{n+1}$.
\STATE
\[
\mfb\leftarrow 1+\frac{\|v_4\|^2}{n_7(n_1n_4-n_5^2)}.
\]
\STATE $W\leftarrow W_0/(\mfa\mfb)$
\STATE $L\leftarrow1-W+W^2/6$;\;\;$U\leftarrow 1-W/2$.
\STATE $t_i\leftarrow \LowerUpperBoundt{}(n_4,n_6,\dot n,L,U)$.\label{line: where LUquadratic is called}
\STATE $s_{i+1} \leftarrow \min\{1,s_i + t_i\} $;
\STATE $\eps\leftarrow\eps_0/\mfa$
\STATE $\displaystyle \tilde{M}\leftarrow \binom{(1-s_{i+1})M_1+s_{i+1}M_2}{{z_i}^*}\in\mathcal{M}_{n+1}(\C)$.
\STATE $v_5\leftarrow g_{i+1}(z_i) = (1-s_{i+1})g(z_i) + s_{i+1} v_1\in \C^n$
\STATE $z_{i+1}\leftarrow z_i-\tilde{M}^{-1}\binom{v_5}{0}\in\C^{n+1}$.
\STATE $\tilde{z}_{i+1}\leftarrow \mbox{\Computeapproxzero{}}(z_{i+1},\eps)$.
\STATE $z_{i+1}\leftarrow \tilde{z}_{i+1}$.
\STATE $i \leftarrow i + 1$.
\ENDWHILE
\STATE $z_* \leftarrow {z}_{i}$.
\end{algorithmic}
\end{algorithm}

We should point out that in our practical implementation of the algorithm lines 13, 14, and 19 as well as lines 26, 27, and 28 correspond to the calls to the subroutine executing one step of Newton's method for a specialization of the system~(\ref{equ: linear homotopy}). We break this up into smaller steps above for the purpose of the complexity analysis performed in Subsection \ref{subsec:bitcomp}.

\begin{remark}\label{rmk:bound2lgeneral}
From (\ref{eq:UminusLBoundbis}) and (\ref{eq:boundml}), for every $i\geq0$ the number of iterations of \LowerUpperBoundt{} at Step 23 is at most
\[
O\left(\log_2 \max\left(1,\frac{\|f-g\|nd^3\max\{\mu(f_t,\zeta_t):0\leq t\leq 1\}}{\min\{\|f_t\|:0\leq t\leq 1\}}\right)\right).
\]
\end{remark}

\section{Complexity analysis}

In this section we analyze the bit complexity of \TrackReallyLinearHomotopypractical{}. Given a rational number $p/q\in\Q$, $gcd(p,q)=1$, the {\em bit length} of $p/q$ is defined as
\[
{\rm bl}(p/q)=\log_2(\max{|p|,|q|})+1.
\]
We also define ${\rm bl}(0)=1$. Note that ${\rm bl}(p/q)$ is a (tight) upper bound for the number of binary digits required to write up $p$ or $q$. writing $p/q$ thus takes at most $2{\rm bl(p/q)}$ bits.

Recall that an algorithm (i.e. a Turing machine) is said to have running time polynomial on quantities $c_1(x),c_2(x),\ldots,c_l(x)$ (where the $c_i(x)$ are quantities depending on the input $x$ of the machine) if there exists a polynomial $p(X)\in\R[X_1,\ldots,X_l]$ such that the running time of the machine on input $x$ is bounded above by $p(c_1(x),\ldots,c_l(x))$. A convenient notation is the following: given some function $f(x)$ depending on the input $x$, we say that
\[
f(x)\leq (c_1(x),\ldots,c_l(x))^{O(1)}
\]
if a polynomial $p$ exists such that $f(x)\leq p(c_1(x),\ldots,c_l(x))$ for all possible input $x$. If a machine has running time which is polynomial in the (bit) size of its input, that is if the running time of the machine is $input\_size^{O(1)}$ then we say that the machine works in polynomial time. The reader does not need be very familiar with the concepts of computational complexity or Turing machine model to understand this section. However, we quote \cite[Introduction]{BlCuShSm98} and its references for a brief yet illustrating introduction to the different concepts of algorithms, and \cite{Papadimitriou94} for a systematic introduction to Turing machines and their complexity.

When it comes to adding or multiplying rational numbers, there exist smart ways of designing the operations which can notoriously speed up the elementary algorithms, see for example \cite{CormenLeisersonRivest1990}. However, we will not search for the optimal upper bounds on the complexity of our algorithm, because our intention is just to prove that it is polynomial in certain quantities as claimed in Theorem \ref{th:mainintro}. We just recall from \cite{CormenLeisersonRivest1990} that $k$ arithmetic operations\footnote{By a.o. we mean an operation of the form $+,-,\times,/$, or a comparison $<,\leq$ or an assignment of a value to a variable, or computation of the integer part of a number.} (a.o. from now on) can be performed on rational inputs of bit length at most $h$, in time which is polynomial in $k$ and $h$, that is in time $(kh)^{O(1)}$, and the result of this sequence of a.o. is a rational number $r\in\Q$ such that ${\rm bl}(r)\leq (kh)^{O(1)}$.

Given a vector $v\in\Q[\ii]^k$, we define its bit length as
\[
{\rm bl}(v)=\max\{{\rm bl}(a_i),{\rm bl}(b_i):v=(a_1+\ii b_1,\ldots, a_k+\ii b_k)\}.
\]

\subsection{Bit complexity of \Signcompute{}}
Let $h$ be an upper bound for the bit length of the input $(a,b,c,t,L,U)$ of \Signcompute{}. The algorithm performs a fixed number of arithmetic operations on the rational numbers which are its input. Hence, the bit complexity of \Signcompute{} is $h^{O(1)}$.
\subsection{Bit complexity of \LowerUpperBoundt{}}\label{sec:complexityLowerUpperBoundt}
Let $h$ be an upper bound for the bit length of the input $(a,b,c,L,U)$ of \LowerUpperBoundt{}. Until Step $11$, \LowerUpperBoundt{} performs a fixed number of arithmetic operations on the rational numbers which are its input (including two applications of \Signcompute{}). The bit complexity of \LowerUpperBoundt{} until Step $11$ is thus $h^{O(1)}$. Each of the loops starting at Step $12$ also performs a fixed number of arithmetic operations, but now the bit length of the number $t_2$ invoked in \Signcompute{} at line $19$ grows with each loop. More precisely, after $i$ iterations,
\[
{\rm bl}(t_2)\leq O(i),
\]
and thus the maximum bit length in all the numbers appearing at the algorithm in the $i$--th loop is $(h+i)^{O(1)}$. The total bit complexity is thus
\[
O(h)+\sum_{i=1}^{\sharp\text{loops}}(h+i)^{O(1)}\leq (h+\sharp\text{loops})^{O(1)}.
\]
From Remark \ref{rmk:bound2lgeneral}, during an application of \TrackReallyLinearHomotopypractical{}
\[
\sharp\text{loops}\leq
O\left(\log_2 \max\left(1,\frac{\|f-g\|nd^3\max\{\mu(f_t,\zeta_t):0\leq t\leq 1\}}{\min\{\|f_t\|:0\leq t\leq 1\}}\right)\right)
\]
Thus, the bit complexity of \LowerUpperBoundt{} on inputs of bit length at most $h$ is, during an application of \TrackReallyLinearHomotopypractical{}, at most
\[
\left(h+\log_2 \max\left(1,\frac{\|f-g\|nd^3\max\{\mu(f_t,\zeta_t):0\leq t\leq 1\}}{\min\{\|f_t\|:0\leq t\leq 1\}}\right)\right)^{O(1)}.
\]
\subsection{Bit complexity of \Computeapproxzero{}}\label{sec:complexityComputeapproxzero}
Let $h$ be an upper bound for the bit length of the input $(z,\eps)$ of \Computeapproxzero{}. Steps $1$ to $6$ of \Computeapproxzero{} perform $O(n)$ a.o. on inputs of bit length bounded by $h$ and thus these steps take time
\[
(nh)^{O(1)},
\]
which is also a bound for the bit length of $x$ (in the notations of Algorithm \Computeapproxzero{}). The number of iterations the algorithm will perform is then at most
\[
O(\log_2\|x\|)\leq O(\log_2(nh)).
\]
at each step the bit length of $\alpha$ increases by a factor of $4$, and checking the stopping criterion can be done in $(nh/\log_2(\eps))^{O(1)}$. Hence the total bit complexity of the while loop is
\[
\sum_{i=1}^{\sharp\text{loops}}((nh/\log_2(\eps))^{O(1)}+O(i))\leq (nh\log_2(\eps)+\sharp\text{loops})^{(O(1)}\leq (nh\log_2(\eps))^{O(1)}.
\]
Step $11$ can then be done in $(nh)^{O(1)}$. Thus, the total bit complexity of \Computeapproxzero{} is $(nh\log_2(\eps))^{O(1)}$.

\subsection{Bit complexity of \TrackReallyLinearHomotopypractical{}}\label{subsec:bitcomp}
Let $h$ be an upper bound for the bit length of the input $(f,g,z_0)$ of \TrackReallyLinearHomotopypractical{}. Let $S>0$ be the number of non-zero monomials in the dense representations of $f$ and $g$. We assume that
\[
\mu_{max}=\max\{\mu(f_t,\zeta_t):0\leq t\leq 1\}<\infty,
\]
which indeed implies that $\mathcal{C}_0<\infty$ and by Theorem \ref{th:stepsizereallylinear} we know that \TrackReallyLinearHomotopypractical{} actually produces an approximate zero of $f$. We now analyze the operations performed in each step of \TrackReallyLinearHomotopypractical{}.

\begin{enumerate}
\item The operations before the while loop:
\begin{itemize}
\item Steps $2,3,4$: two squared--norm computations and one inner product computation. That is $O(S)$ a.o. with rationals of bit length $\max\{h,l\}$ where $l$ is an upper bound for the bit length of the multinomial coefficients $\binom{d_i}{\alpha_i}$ which appear in the definition of Bombieri--Weyl's product (see Section \ref{sec:Hd}). Note that $l\leq\log(d!)\leq d^{O(1)}$. Thus, $\max\{h,l\}\leq (h+d)^{O(1)}\leq(hd)^{O(1)}$ and the bit complexity of these steps is at most $(Shd)^{O(1)}$. The numbers they produce have bit length $(Shd)^{O(1)}$ as well.
\item Steps $1,5,6,7$: a constant number of a.o. with rationals of bit length $(Shd)^{O(1)}$ is again $(Shd)^{O(1)}$ (and the numbers produced have the bit length bounded by the same quantity).
\end{itemize}
\item Step $8$ (number of loops): from Theorem \ref{th:stepsizereallylinear} and Lemma \ref{lem:frobenius}, the number of loops is at most $\lceil \constantForHomotopyComplexity\sqrt{n+1}d^{3/2}\mathcal{C}_0\rceil$, where $\mathcal{C}_0$ is the length of the path $(f_t,\zeta_t)$ in the condition metric. For counting the bit complexity of each loop, let $h_i$ be $h$ or the maximum bit length of the rational numbers $s_i,z_i,t_i$ (whichever is greater), and let $h_{max}=\max\{h_i\}$ (we will prove latter that $h_{max}<\infty$). The bit complexity of the $i$--th loop is bounded as follows.
\begin{itemize}
\item Steps $9,10,11$: a constant number of a.o. with rationals of bit length $(h_i)^{O(1)}$ is again $(h_i)^{O(1)}$.
\item Step $12$: computation of the squared norm of a $\C^{n+1}$ vector with rational coordinates of bit length bounded by $h_i$: bit complexity $(nh_i)^{O(1)}$ and $n_7$ has bit length at most $(nh_i)^{O(1)}$, as well.
\item Step $13$: computation of the derivative matrices of $f$ and $g$ at $z_i$, which is $(nSdh_i)^{O(1)}$ using the elementary evaluation method (see \cite{BaurStrassen1983} for a faster but more complicated one), and the bit length of the numbers is at most $(nSdh_i)^{O(1)}$.
\item Step $14$: addition of two $n\times(n+1)$ matrices with rational entries of bit length at most $(nSdh_i)^{O(1)}$ is $(nSdh_i)^{O(1)}$, then an inverse matrix computation is $(nSdh_i)^{O(1)}$ using modular techniques\footnote{Exact linear algebra is a large research field, see {\tt http://linalg.org/people.html} for a list of people working on the subject, as well as software and research articles.}. Indeed, computing of the inverse is equivalent to solving $n+1$ systems of equations with rational coefficients. Each of these systems can be first normalized to systems with integer coefficients of size $(nSdh_i)^{O(1)}$, which (according to, e.g., \cite{Dixon82}) can be solved in time $(nSdh_i)^{O(1)}$. The total bit complexity of this step is thus $(nSdh_i)^{O(1)}$.
\item Step $15$: $O(n^2\log_2(d))$ arithmetic operations (the $\log_2 d$ in this formula is needed to compute $n_7^{d_{l+1}-1}$) with numbers of bit length $(nSdh_i)^{O(1)}$ is again $(nSdh_i)^{O(1)}$, and the bit length of $a$ is again bounded by $(nSdh_i)^{O(1)}$.
\item Steps $16,17$: computation of $f(z_i)$ and $g(z_i)$ is $(nSdh_i)^{O(1)}$ because that is a bound for the bit length of the rational numbers appearing in the monomial expansion of $f,g$ and also for the bit length of the coordinates of $z_i$. There are also a constant number of a.o. which is again $(nSdh_i)^{O(1)}$.
\item Step $18$: a constant number of a.o. is again $(nSdh_i)^{O(1)}$.
\item Step $19$: a matrix--vector product, $O(n^2)$ a.o. with rationals of bit length bounded by $(nSdh_i)^{O(1)}$, is again $(nSdh_i)^{O(1)}$.
\item Steps $20,21,22$: a constant number of a.o. is again $(nSdh_i)^{O(1)}$.
\item Step $23$: an application of \LowerUpperBoundt{} with input data whose bit length is bounded by $(nSdh_i)^{O(1)}$, according to Section~\ref{sec:complexityLowerUpperBoundt} costs
\[
\left((nSdh_i)^{O(1)}+\log_2 \max\left(1,\frac{\|f-g\|nd^3\mu_{max}}{\min\{\|f_t\|:0\leq t\leq 1\}}\right)\right)^{O(1)}.
\]
By hypothesis, the output of \LowerUpperBoundt{} has the bit length bounded by $h_{i+1}\leq h_{max}$.
\item Step $24$: a constant number of a.o. is again $(nSdh_{max})^{O(1)}$.
\item Step $25$: a division of two rational numbers of the respective bit lengths $(nSdh_{max})^{O(1)}$ and
\begin{equation}\label{eq:bla}
{\rm bl}(\mfa)={\rm bl}(\tilde\varphi_i^2)\underset{(\ref{eq:varphivsmui})}{\leq} O(\sqrt{n}\log_2\mu_{max})
\end{equation}
costs $(nSdh_{max}\log_2\mu_{max})^{O(1)}.$
\item Step $26$: adding two $n\times (n+1)$ matrices, $O(n^2)$ a.o. with rationals of bit length $(nSdh_{max}\log_2\mu_{max})^{O(1)}$ has bit complexity $(nSdh_{max}\mu_{max})^{O(1)}$.
\item Step $27$: as in Step $17$, this takes time $(nSdh_i)^{O(1)}$.
\item Step $28$: solving a system of equations and adding two vectors with bit lengths bounded by $(nSdh_{max}\log_2\mu_{max})^{O(1)}$ is again $(nSdh_{max}\mu_{max})^{O(1)}$ according to~\cite{Dixon82}.
\item Step $29$: an application of \Computeapproxzero{} with input whose bit length is bounded by $(nSdh_{max}\log_2\mu_{max})^{O(1)}$. From Section \ref{sec:complexityComputeapproxzero}, this has bit complexity $(nSdh_{max}\log_2\mu_{max})^{O(1)}$.
\item Step $30,31$: a constant number of a.o. with rationals of bit length $(nSdh_{max}\log_2\mu_{max})^{O(1)}$ is $(nSdh_{max}\log_2\mu_{max})^{O(1)}$.
\end{itemize}
\item Step $33$: One a.o. is again $(nSdh_{max}\log_2\mu_{max})^{O(1)}$.
\end{enumerate}
The bit complexity of \TrackReallyLinearHomotopypractical{} is thus
\[
\left((nSdh_{max}\log_2\mu_{max})^{O(1)}+\log_2 \max\left(1,\frac{\|f-g\|nd^3\mu_{max}}{\min\{\|f_t\|:0\leq t\leq 1\}}\right)\right)^{O(1)}\mathcal{C}_0,
\]
where $h_{max}$ is the maximum of $h$ and the bit lengths of $s_i,t_i$ and $z_i$. Now, all the $s_i$ and $t_i$ are numbers of the form $m/2^l$ where, from Remark~\ref{rmk:bound2lgeneral},
\[
{\rm bl}(m)\leq {\rm bl}(2^l)=l+1\leq O\left(\log_2 \max\left(1,\frac{\|f-g\|nd^3\mu_{max}}{\min\{\|f_t\|:0\leq t\leq 1\}}\right)\right).
\]
Thus, this is also an upper bound for the bit lengths of $s_i$ and $t_i$. As for that of $z_i$, note that from Lemma \ref{lem:findapproxzero} we have that at each step $i\geq1$,
\[
{\rm bl}(z_i)\leq O\left(\log_2\frac{\sqrt{n}}{\eps}\right)\leq O(\log_2(\sqrt{n} \mfa ))\underset{(\ref{eq:bla})}{\leq} O(\log_2(n\mu_{max})).
\]
Hence, we have
\[
h_{max}\leq h+\left(\left(\log_2 \max\left(1,\frac{\|f-g\|nd^3\mu_{max}}{\min\{\|f_t\|:0\leq t\leq 1\}}\right)\right)^2+\log_2(n\mu_{max})\right).
\]
The bit complexity of \TrackReallyLinearHomotopypractical{} is thus linear in $\mathcal{C}_0$ and polynomial in the following quantities:
\begin{itemize}
\item $n,S,d,h$,
\item $\log_2\mu_{max}$,
\item $\log_2(\|f-g\|/\min\{\|f_t\|:0\leq t\leq 1\})$.
\end{itemize}

\section{Proof of Theorem \ref{th:mainintro}}

We first note that \TrackReallyLinearHomotopypractical{}, performs the operations described by \TrackReallyLinearHomotopy{}, except for the use of Frobenius norm instead of operator norm in the computation of $\chi_{i,1}$. This follows directly from the description of the two algorithms and from lemmas \ref{lem:bisection} and \ref{lem:findapproxzero}.

Thus, from Theorem \ref{th:stepsizereallylinear} and Lemma \ref{lem:frobenius}, \TrackReallyLinearHomotopypractical{} has certified output. Moreover, its total bit complexity has been proved in section \ref{subsec:bitcomp} to satisfy the claim of Theorem \ref{th:mainintro}. For the bound on the size of the output, let $i=k$ be the final step of the algorithm. Then, the output $z_{k+1}$ of \TrackReallyLinearHomotopypractical{} is the result of applying \Computeapproxzero{} to some $(z_k,\eps)$ where $z_k\in\Q[\ii]^{n+1}$ and
\[
\eps=\frac{\eps_0}{\mfa}\underset{(\ref{eq:varphivsmuibis})}{\geq}\frac{c_0}{\sqrt{n+1}\mu(g_k,\zeta_k)^2}\underset{(\ref{eq:comparemus})}\geq \frac{c_1}{\sqrt{n}\mu(f,\zeta_{k+1})^2},
\]
$c_0$ and $c_1$ some constants. It follows from Lemma \ref{lem:findapproxzero} that $z_{k+1}$ has integer coordinates of bit length at most $O(\log_2(n\mu(f,\zeta_{k+1})))$, as claimed. The proof is now complete.

\section{Experiments}

Our implementation of Algorithm~\ref{alg:practical_homotopy} has been carried out in the top-level (interpreted) language of {\em Macaulay2}~\cite{M2www}. The exact linear algebra routines and evaluation of polynomials are inherently slow and there are many engineering improvements that can be made to speed up the execution; yet the computation takes reasonable time on the examples of modest size.

While more examples of computation along with the source code of the implementation are available at

{\tt http://people.math.gatech.edu/\~{}aleykin3/RobustCHT/}

\noindent here we describe two experiments. One of them involves a small family of equations, where most of the computation of the length of a homotopy path $\mathcal{C}_0$ can be carried out by hand. The other comes from an application in enumerative geometry and showcases the class of problems that can benefit from the developed certified algorithms.

\subsection{Actual number of steps vs. condition length}
In Lemma \ref{lem:frobenius} we claim that the number of steps (i.e. number of while loops) needed by Algorithm \ref{alg:practical_homotopy} is at most $\lceil  \constantForHomotopyComplexity\sqrt{n+1}d^{3/2}\mathcal{C}_0\rceil$. In this section we consider a simple family of examples parametrized by $m\geq0$ where the value of $\mathcal{C}_0$ can be approximated by quadrature formulas and show how the bounds based on $\mathcal{C}_0$ compare to the actual performance of the algorithm. Note that from (\ref{eq:condlength}) the condition length of a path $(f_t,\zeta_t)\subseteq\CHd\times\Pn$ (with $\zeta_t$ given by a smooth curve of affine representatives) is
\[
\int_0^1\chi_1(t)\sqrt{\frac{\|\dot f_t\|^2}{\|f_t\|^2}-\frac{\Real{(\langle \dot{f}_t,f_t\rangle)}^2}{\|f_t\|^4}+\frac{\|\dot{\zeta}_t\|^2}{\|\zeta_t\|^2}-\frac{|\langle \dot{\zeta}_t,\zeta_t\rangle|^2}{\|\zeta_t\|^4}},
\]
were
\[
\chi_1(t)=\left\|\binom{Df_t({\zeta_t})}{{\zeta_t}^*}^{-1}
\begin{pmatrix}\sqrt{d_1}\|f_t\|\|{\zeta_t}\|^{d_1-1}& & & \\
& \ddots & &\\
& &\sqrt{d_n}\|f_t\|\|{\zeta_t}\|^{d_n-1}\\
& & &\|{\zeta_t}\|\end{pmatrix}\right\|.
\]
In general, it is extremely hard to compute a priori $\mathcal{C}_0$ (even approximately). We consider here the simple case
\[
f_t(x_0,x_1)=x_1^2-(1+mt)x_0^2,\qquad \zeta_t=(1,\sqrt{1+mt})^T.
\]
Let $s=1+mt$. We can easily compute:
\[
\|f_t\|^2=1+s^2;\qquad \dot f_t=-mx_0^2;\qquad\|\dot f_t\|^2=m^2;\qquad \langle \dot{f}_t,f_t\rangle=ms;
\]
\[
\dot\zeta_t=\left(0,\frac{m}{2\sqrt{s}}\right)^T;\qquad \|\zeta_t\|^2=1+s;\qquad \|\dot\zeta_t\|^2=\frac{m^2}{4s};\qquad|\langle\dot\zeta_t,\zeta_t\rangle|^2=\frac{m^2}{4};
\]
Thus,
\[
\sqrt{\frac{\|\dot f_t\|^2}{\|f_t\|^2}-\frac{\Real{(\langle \dot{f}_t,f_t\rangle)}^2}{\|f_t\|^4}+\frac{\|\dot{\zeta}_t\|^2}{\|\zeta_t\|^2}-\frac{|\langle \dot{\zeta}_t,\zeta_t\rangle|^2}{\|\zeta_t\|^4}}=
\]
\[
 m\sqrt{\frac{1}{1+s^2}-\frac{s^2}{(1+s^2)^2}+\frac{1}{4s(1+s)}-\frac{1}{4(1+s)^2}}=m\sqrt{\frac{1}{(1+s^2)^2}+\frac{1}{4s(1+s)^2}}.
\]
On the other hand,
\[
\begin{pmatrix} Df_t(\zeta_t)\\\zeta_t^*\end{pmatrix}^{-1}=\begin{pmatrix}-2s&2\sqrt{s}\\1 &\sqrt{s}\end{pmatrix}^{-1}= \begin{pmatrix}\frac{-1}{2(1+s)}&\frac{1}{1+s}\\ \frac{1}{2\sqrt{s}(1+s)}&\frac{\sqrt{s}}{1+s}\end{pmatrix}
\]
\[
\chi_1(t)=\left\|\begin{pmatrix}\frac{-1}{2(1+s)}&\frac{1}{1+s}\\ \frac{1}{2\sqrt{s}(1+s)}&\frac{\sqrt{s}}{1+s}\end{pmatrix}
\begin{pmatrix}\sqrt{2}\sqrt{1+s^2}\sqrt{1+s}&0\\
0&\sqrt{1+s}\end{pmatrix}\right\|=
\]
\[
\frac{1}{\sqrt{1+s}}\left\|\begin{pmatrix}\frac{-\sqrt{1+s^2}}{\sqrt{2}}&1\\ \frac{\sqrt{1+s^2}}{\sqrt{2s}}&\sqrt{s}\end{pmatrix}
\right\|=
\frac{\sqrt{1+s^2}}{\sqrt{2s}},
\]
where to get the the last equality we compute the matrix norm by hand. With the change of variables $s=1+mt$ we have then proved that
\[
\mathcal{C}_0(f_t,\zeta_t)=\int_1^{1+m}\frac{\sqrt{1+s^2}}{\sqrt{2s}}\sqrt{\frac{1}{(1+s^2)^2}+\frac{1}{4s(1+s)^2}}\;ds,
\]
It is not an easy task to find this integral exactly, but we can at least try to approximate with some quadrature formula. In Octave-produced Table \ref{table:comparisonC0steps} and Figure \ref{fig:comparisonC0steps} we compare the values of upper and lower bounds
\[
\begin{array}{ccl}
LBound&\leq&\sharp(steps)\leq UBound,\text{ where}\\
LBound&=&28d^{3/2}\mathcal{C}_0(f_t,\zeta_t)\approx 79\mathcal{C}_0,\\
UBound&=&\constantForHomotopyComplexity\sqrt{n+1}d^{3/2}\mathcal{C}_0(f_t,\zeta_t)=316\mathcal{C}_0(f_t,\zeta_t)
\end{array}
\]
for different choices of $m\geq0$ and the number of steps performed by our algorithm to follow the homotopy $f_t$.
\begin{center}
\begin{table}[h]
\caption{Comparison of the bound of number of steps given by Lemma \ref{lem:frobenius} and the actual number of steps in the example given by $f_t=x_1^2-(1+mt)x_0^2$.}
\label{table:comparisonC0steps}
\begin{tabular} {|c|c|c|c|c|}
\hline
m&LB&steps&UB&UB/steps\\
\hline
10& 31& 184& 357& 1.95\\
\hline
20& 38& 217& 435& 2.01\\
\hline
30& 42& 237& 480& 2.03\\
\hline
40& 45& 250& 512& 2.05\\
\hline
50& 47& 260& 537& 2.07\\
\hline
60& 49& 269& 558& 2.08\\
\hline
70& 50& 276& 575& 2.08\\
\hline
80& 52& 282& 590& 2.09\\
\hline
90& 53& 288& 603& 2.1\\
\hline
100& 54& 292& 615& 2.11\\
\hline
1000& 77& 395& 872& 2.21\\
\hline
2000& 84& 426& 949& 2.23\\
\hline
3000& 88& 446& 995& 2.23\\
\hline
4000& 91& 457& 1027& 2.25\\
\hline
5000& 93& 468& 1052& 2.25\\
\hline
10000& 100& 499& 1129& 2.26\\
\hline
20000& 106& 530& 1207& 2.28\\
\hline
30000& 110& 547& 1252& 2.29\\
\hline
\end{tabular}
\end{table}
\end{center}

 \begin{figure}[h]
      {
      \centerline{
        \includegraphics[scale=1]{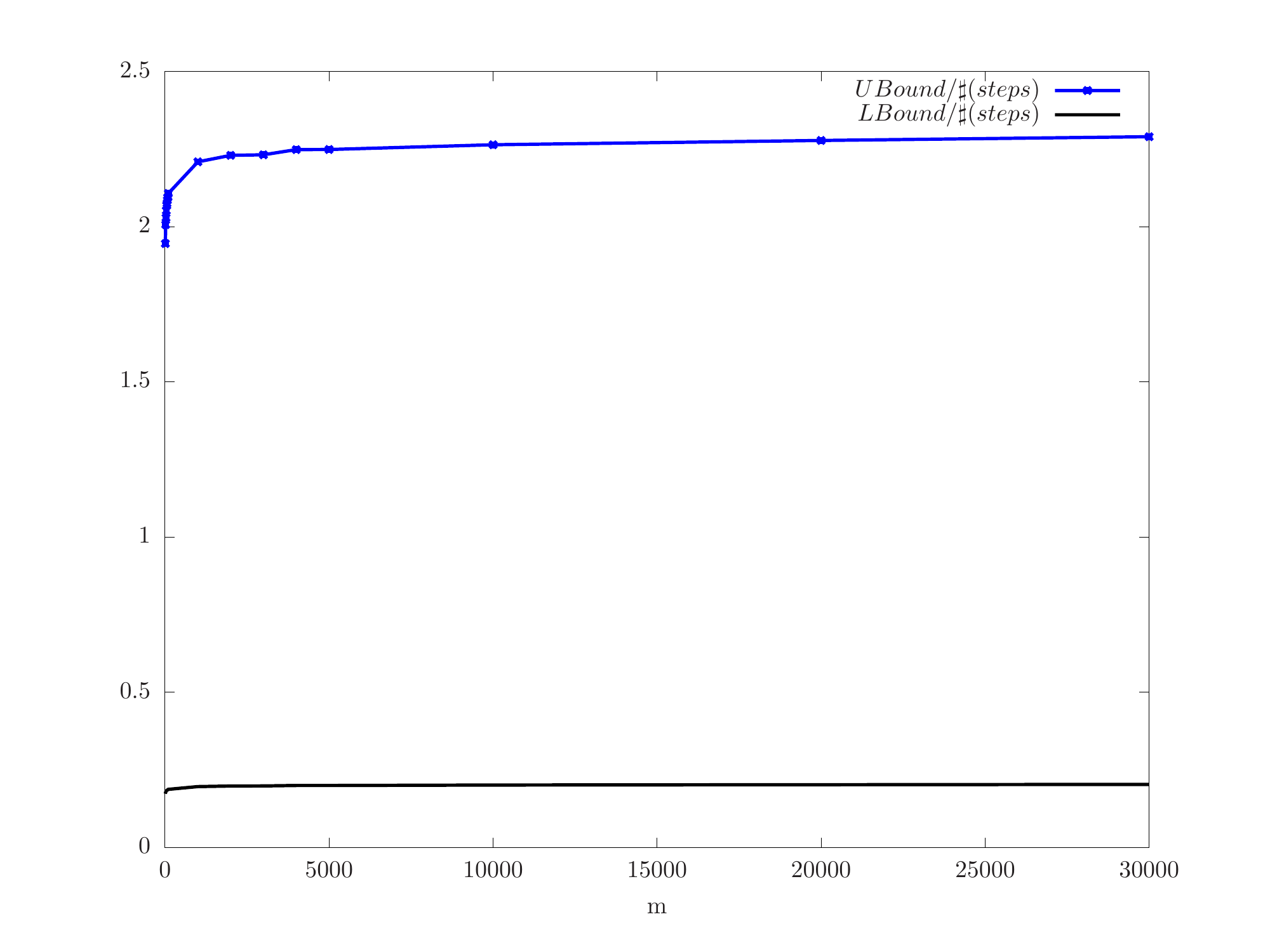}}
      }
\caption{Comparison of the ratio between the actual number of steps and its lower and upper bound.}
    \label{fig:comparisonC0steps}
    \end{figure}

\subsection{An application to a problem in Schubert calculus}

The computations of~\cite{Leykin-Sottile:HoG} confirmed the conjecture saying that the Galois group of a simple Schubert problem is the full symmetric group for ``small'' Grassmannians. These results produced using heuristic homotopy continuation methods take us far beyond the limitations of the symbolic methods.

Table~\ref{table:Galois}, a copy of~\cite[Table~1]{Leykin-Sottile:HoG}),  shows the number of solutions for the largest problem on $G(k,n)$ and the number of permutations found in the Galois group by the algorithm sufficient to generate the full symmetric group.
At the present all computations can be done within one day with a heuristic homotopy tracker employed.

\begin{table}[h]
 \begin{tabular}{|c||c|c|c|c|c|c|c|}\hline
  $k,n$ &{\bf 2,4}&{\bf 2,5}&{\bf 2,6}&2,7&2,8&2,9&2,10\\\hline
  solutions&{\bf 2}&{\bf 5}&{\bf 14}&42&132&429&1430\\\hline
  permutations&{\bf 4}&{\bf 6}&{\bf 5}&6&7&4&7\\\hline
 \end{tabular}\vspace{8pt}
 \begin{tabular}{|c||c|c|c|c|c||c|c|c|}\hline
  $k,n$ &3,5&3,6&3,7&3,8&{3,9}&4,6&4,7&4,8\\\hline
  solutions&5&42&462&6006&{17589}&14&462&8580\\\hline
  permutations&4&4&5&6&{7}&5&5&7\\\hline
 \end{tabular}\vspace{8pt}
 \caption{Galois group computation for simple Schubert problems in $G(k,n)$.}
 \label{table:Galois}\vspace{-10pt}
\end{table}

This problem falls naturally in the class where the certified algorithms of this paper can be applied. With the current implementation the algorithm of this paper can provide the {\em status of a theorem} to all of the computational results on up to $\Gr(2,6)$: the cases that can be certified within a day appear in bold in Table~\ref{table:Galois}.

The corresponding runs of the algorithm for $\Gr(2,6)$ involve tracking homotopies for six polynomial equations following the paths in $\P^6$ and have input, output, and all intermediate approximate zeroes defined over Gaussian integers $\Z[i]$. Due to the use of our Algorithm \Computeapproxzero\; to reduce the size of the integers in the intermediate steps, in this relatively large computation we do not encounter integers longer than {\em six} decimal digits amongst the coordinates of {\em all} approximate zeroes computed along all homotopy paths.

Let us remark that the largest certifiable case is already beyond the reach of purely symbolic algorithms (the problem with 14 solutions in $\Gr(2,6)$ is characterized as ``not computationally feasible'' in \cite{Billey-Vakil}).  There are several ways to push the frontier of provable results further. One is a low-level optimized implementation of our algorithm. Another is using a fast heuristic homotopy tracker to find the ``interesting'' paths (e.g., the ones that do not lead to a redundant permutation in the Galois group computation), break them up into a union of smaller pieces, and then execute a certified homotopy tracker for every small piece. The last step is trivially parallelizable and can be sped up in practice by using distributed computing.

\section{Proof of Theorem \ref{th:onestep}}\label{sec:proofth2}
We recall first two lemmas \cite[Lemma 4 and Lemma 5]{Beltran2009NM}. The second of these two lemmas is recalled here in a less general version than the original.
\begin{lemma}\label{lem:homoto1b}
Let $h_0,h\in\S$, $v\in\CHd$, $z_0,z\in\Pn$. Assume that $\chi_1(h_0,z_0)<+\infty$. Assume moreover that
\[
d_R(z_0,z)\leq\frac{\hat{a} }{d^{3/2}\chi_1(h_0,z_0)},
\]
\[
d_\S(h_0,h)\leq\frac{3\hat{a} }{2d^{3/2}\chi_1(h_0,z_0)},
\]
for some $\hat{a}<1/\sqrt{2}$. Then,
\[
\frac{\chi_1(h_0,z_0)}{1+\sqrt{2}\hat{a}}\leq \chi_1(h,z)\leq\frac{\chi_1(h_0,z_0)}{1-\sqrt{2}\hat{a} }\text{ and}
\]
\[
\varphi(h_0,v,z_0)\frac{(1-\sqrt{2}\,\hat{a} )^{\sqrt{2}}}{1+\sqrt{2}\,\hat{a}
}\leq\varphi(h,v,z)\leq\frac{\varphi(h_0,v,z_0)}{(1-\sqrt{2}\,\hat{a} )^{1+\sqrt{2}}}.
\]
\end{lemma}
\begin{lemma}\label{lem:homoto2}
Let $t\rightarrow h_s\in\S$, $0\leq s\leq T$ be a piece of a great circle in $\S$, parametrized by arc--length. Let $\eta_0\in\Pn$ be a projective zero of $h_0$ such that
$\mu(h_0,\eta_0)<+\infty$. Assume that
\[
T\leq \frac{1}{Pd^{3/2}\hat\varphi},\text{ where }\hat\varphi=\varphi(h_0,\dot h_0,\eta_0).
\]
Then, for $0\leq s<T$, $\eta_0$ can be continued to a zero $\eta_s\in\Pn$ of $h_s$ in such a way that $s\rightarrow
\eta_s$ is a $C^{1+Lip}$ curve. Moreover, consider the curve $s\rightarrow(h_s,\dot h_s,\eta_s)$, $0\leq s< T$.
Then, the following inequalities hold for every $s\in[0,T]$:
\[
\frac{\hat\varphi}{1+P\,d^{3/2}\hat\varphi s}\leq\varphi(h_s,\dot h_s,\eta_s)\leq \frac{\hat\varphi}{1-P\,d^{3/2}\hat\varphi s},
\]
\[
d_R(\eta_0,\eta_s)\leq
\frac{1}{\sqrt{2}d^{3/2}\chi_1(h_0,\zeta_0)}\left(1-\left(1-Pd^{3/2}\hat\varphi s\right)^{\sqrt{2}/P}\right),
\]
\[
d_\S(h_0,h_s)\leq\frac{1}{d^{3/2}H} \log\frac{1}{1-d^{3/2}H\chi_2(h_0,\dot h_0,\eta_0)s}
\]
\end{lemma}

Now we proceed to the proof of Theorem \ref{th:onestep}. Recall that we have defined $T=d_\S\left(\frac{g}{\|g\|},\frac{f}{\|f\|}\right)$. Consider the path
\begin{equation}
s\rightarrow h_s=\frac{g}{\|g\|}\cos(s)+\frac{\frac{f}{\|f\|}-\Real(\langle \frac{f}{\|f\|},\frac{g}{\|g\|}\rangle) \frac{g}{\|g\|}}{\sqrt{1-\Real(\langle
\frac{f}{\|f\|},\frac{g}{\|g\|}\rangle)^2}}\sin(s),\;\;\;s\in\left[0,T\right],
\end{equation}
That is, $h_s$ is the arc--length parametrization of the short portion of the great circle joining $g/\|g\|$ and $f/\|f\|$. Note that $\dot h_0=\dot g$ as was defined in Theorem \ref{th:onestep}.

Let
\[
\hat{\chi}_{1}=\chi_1(g,\zeta_0)=\mu(g,\zeta_0),\hat{\chi}_{2}=\chi_2(g,\dot g,\zeta_0),\hat{\varphi}=\varphi(g,\dot g,\zeta_0).
\]
From (\ref{eq:hypz0}) and Lemma \ref{lem:homoto1b} we get
\begin{equation}\label{eq:varphicompareb}
 \hat{\varphi}\frac{(1-\sqrt{2}u_0/2)^{\sqrt{2}}}{1+\sqrt{2}u_0/2}\leq\varphi\leq
\frac{\hat{\varphi}}{(1-\sqrt{2}u_0/2)^{1+\sqrt{2}}}.
\end{equation}
From (\ref{eq:22NM}) and (\ref{eq:varphihatvsmuprev}), we have:
\begin{equation}\label{eq:varphihatvsmu}
\hat{\varphi}\leq\mu(g,\zeta_0)\|\dot g\|\sqrt{1+\mu(g,\zeta_0)^2}\leq\sqrt{2}\mu(g,\zeta_0)^2,
\end{equation}
where for the last equality we have used that $\|\dot h_0\|=1$ and $\mu(g,\zeta)\geq1$. The second inequality of (\ref{eq:varphicompareb}), together with (\ref{eq:varphihatvsmu}), then implies (\ref{eq:varphivsmu}).

It also follows that
\begin{equation}\label{eq:stepb}
T\underset{(\ref{eq:distancefgteor})}\leq \frac{c}{Pd^{3/2}\varphi}\underset{(\ref{eq:varphicompareb}),(\ref{eq:cb})}\leq \frac{c'}{Pd^{3/2}\hat{\varphi}}< \frac{1}{d^{3/2}\hat{\varphi}}.
\end{equation}
Thus, Lemma \ref{lem:homoto2} applies and we conclude that $\zeta_0=\eta_0$ can be continued to $\eta_s\in\S(\C^{n+1})$, a zero of $h_s$, for $0\leq s\leq T$. Now, note that $h_s$ is a reparametrization $s=s(t)$ of the projection of $f_t=(1-t)g+tf$ on $\S$. That is, $h_s=f_{t(s)}/\|f_{t(s)}\|$. Hence, $\zeta_0$ can be continued to $\zeta_t=\eta_{s(t)}$, a zero of $f_t$ as claimed in Theorem \ref{th:onestep}. Note that $\zeta_1=\eta_{s(1)}=\eta_T$. Moreover, Lemma \ref{lem:homoto2} and (\ref{eq:22NM}) also imply that, for $0\leq s\leq T$,
\begin{equation}\label{eq:varphifromlem5NM}
\frac{\hat\varphi}{1+Pd^{3/2}\hat\varphi s}\leq \mu(h_s,\eta_s)\|(\dot h_s,\dot \eta_s)\|_{T_{(h_s,\eta_s)}\S\times\Pn}\leq \frac{\hat\varphi}{1-Pd^{3/2}\hat\varphi s},
\end{equation}
and that
\[
 d_R(\zeta_0,\zeta_1)=d_R(\eta_0,\eta_T)\leq\frac{1}{\sqrt{2}d^{3/2}\hat{\chi}_{1}}\left(1-\left(1-Pd^{3/2}
\hat{\varphi}\frac{c}{Pd^{3/2}\varphi} \right)^{\sqrt{2}/P}\right)\underset{(\ref{eq:varphicompareb}), (\ref{eq:cb})}\leq
\]
\begin{equation}\label{eq:te1b}
\frac{1}{\sqrt{2}d^{3/2}\hat{\chi}_{1}}\left(1-\left(1-c'\right)^{\sqrt{2}/P}\right)\underset{(\ref{eq:cA2b})}{=} \frac{a}{\sqrt{2}d^{3/2}\hat\chi_1}\underset{(\ref{eq:22NM})}{=}  \frac{a}{\sqrt{2}d^{3/2}\mu(g,\zeta_0)}.
\end{equation}

We have seen (\ref{eq:stepb}) that $T\leq c'(Pd^{3/2}\hat{\varphi})^{-1}$. Now, $\hat{\varphi}=\hat{\chi_1}\hat{\chi_2}$ and $\hat{\chi_2}\geq\|\dot h_s\|=1$, which implies
\begin{equation*}
T\leq \frac{c'}{Pd^{3/2}\hat{\chi}_1}\underset{(\ref{eq:22NM})}{=}\frac{c'}{Pd^{3/2}\mu(g,\zeta_0)}\underset{(\ref{eq:inequalityb})}{\leq}\frac{3a}{2\sqrt{2}d^{3/2}\mu(g,\zeta_0)}.
\end{equation*}

Note that this last inequality, (\ref{eq:te1b}) and Lemma \ref{lem:homoto1b} imply
\begin{equation}\label{eq:comparemus}
 \frac{\mu(g,\zeta_0)}{1+a}\leq \mu(f,\zeta_1)\leq\frac{\mu(g,\zeta_0)}{1-a}.
\end{equation}
Thus,
\begin{equation}
 \mu(g,\zeta_0)\geq (1-a)\mu(f,\zeta_1),
\end{equation}
and hence

\begin{equation}\label{eq:te4b}
 T\leq \frac{c'}{Pd^{3/2}(1-a)\mu(f,\zeta_1)}\underset{(\ref{eq:inequalityb})}{\leq}\frac{3\delta u_0}{2d^{3/2}\mu(f,\zeta_1)}.
\end{equation}

Note that
\begin{align*}
d_R(z_0,\zeta_1)\mu(f,\zeta_1)&\leq (d_R(z_0,\zeta_0)+d_R(\zeta_0,\zeta_1))\mu(f,\zeta_1)\\
&\underset{(\ref{eq:hypz0}),(\ref{eq:te1b})}{\leq}\left(\frac{u_0}{2d^{3/2}\mu(g,\zeta_0)}+\frac{a}{\sqrt{2}d^{3/2}\mu(g,\zeta_0)}\right)\mu(f,\zeta_1)\\
&\underset{(\ref{eq:comparemus})}{\leq}\left(\frac{u_0}{2d^{3/2}}+\frac{a}{\sqrt{2}d^{3/2}}\right)\frac{1}{1-a}.
\end{align*}
Our choice of $a$ is such that the right--hand term in this last equation is at most $\delta u_0/d^{3/2}$. Hence,
we have
\begin{equation}\label{eq:appzeroproved2teor}
 d_R(z_0,\zeta_1 )\mu(f,\zeta_1 )\leq \frac{\delta u_0}{d^{3/2}}.
\end{equation}
From (\ref{eq:appzeroproved2teor}) and (\ref{eq:te4b}), Lemma \ref{lem:homoto1b} then yields
\begin{equation}\label{eq:rrrb}
 \frac{\mu(f,\zeta_1)}{1+\sqrt{2}\delta u_0}\leq \chi_1\leq \frac{\mu(f,\zeta_1)}{1-\sqrt{2}\delta u_0}.
\end{equation}
Moreover, from Lemma \ref{prop:aptproj}, (\ref{eq:appzeroproved2teor}) implies that $z_0$ is an approximate zero of $f$ with
associated zero $\zeta_1$. In particular,
\begin{equation}
 d_R(N_\P(f)(z_0),\zeta_1 )\leq\frac{d_R(z_0,\zeta_1 )}{2}\leq \frac{\delta u_0}{2d^{3/2}\mu(f,\zeta_1)}.
\end{equation}
From this and (\ref{eq:appztildeteor}) we have
\[
 d_R(\tilde{z},\zeta_1 )\leq d_R(\tilde{z},N_\P(f)(z_0))+d_R(N_\P(f)(z_0),\zeta_1 )\leq
\]
\[
\frac{(1-\delta)u_0}{2d^{3/2}(1+3\delta u_0/2) \chi_{1}}+\frac{\delta u_0}{2d^{3/2}\mu(f,\zeta_1 )}\leq
\]
\[
\frac{(1-\delta)u_0}{2d^{3/2}(1+\sqrt{2}\,\delta u_0) \chi_{1}}+\frac{\delta u_0}{2d^{3/2}\mu(f,\zeta_1 )}\underset{(\ref{eq:rrrb})}{\leq}
\]
\[
\frac{u_0}{2d^{3/2}\mu(f,\zeta_1 )}\left((1-\delta)+\delta\right)= \frac{u_0}{2d^{3/2}\mu(f,\zeta_1 )},
\]
proving (\ref{eq:teorconclussion}).

As for (\ref{eq:teorconc2}), first note that from Lemma \ref{lem:auxiliar} below,
\begin{equation}\label{eq:auxiliar}
\mathcal{C}_0(f_t,\zeta_t)=\int_0^T\mu(h_s,\eta_s)\|(\dot h_s,\dot \eta_s)\|_{T_{(h_s,\eta_s)}\S\times\Pn}\,ds.
\end{equation}

Now, this last equality and (\ref{eq:varphifromlem5NM}) imply:
\[
\mathcal{C}_0(f_t,\zeta_t)\geq\int_0^T\frac{\hat\varphi}{1+Pd^{3/2}\hat\varphi s}\,ds=
\]
\[
\frac{\ln(1+Pd^{3/2}\hat\varphi T)}{Pd^{3/2}}=\hat\varphi T \frac{\ln(1+Pd^{3/2}\hat\varphi T)}{Pd^{3/2}\hat\varphi T}.
\]
Because $\ln(1+t)/t$ is a decreasing function of $t>0$ and
\[
Pd^{3/2}\hat\varphi T\underset{(\ref{eq:varphicompareb})}{\leq} Pd^{3/2}\frac{1+\sqrt{2}u_0/2}{(1-\sqrt{2}u_0/2)^{\sqrt{2}}}\varphi T\underset{( \ref{eq:distancefgteor})}{\leq} c \frac{1+\sqrt{2}u_0/2}{(1-\sqrt{2}u_0/2)^{\sqrt{2}}}\leq c',
\]
we conclude that
\begin{equation}\label{eq:cotainth}
\mathcal{C}_0(f_t,\zeta_t)\geq \hat\varphi T\frac{\ln(1+c')}{c'}\underset{(\ref{eq:varphicompareb})}{\geq}
\end{equation}
\[
\varphi T\frac{(1-\sqrt{2}u_0/2)^{1+\sqrt{2}}\ln(1+c')}{c'}.
\]
On the other hand, using again (\ref{eq:auxiliar}), we have
\begin{equation}\label{eq:cotainth2}
\mathcal{C}_0(f_t,\zeta_t)\underset{(\ref{eq:varphifromlem5NM})}{\leq} \int_0^T\frac{\hat\varphi}{1-Pd^{3/2}\hat\varphi s}\,ds=
\end{equation}
\[
\hat\varphi T\frac{\log(1-Pd^{3/2}\hat\varphi T)}{-Pd^{3/2}\hat\varphi T}\leq\hat\varphi T\underset{(\ref{eq:varphicompareb})}{\leq}\varphi T\frac{1+\sqrt{2}u_0/2}{(1-\sqrt{2}u_0/2)^{\sqrt{2}}}.
\]
Note that (\ref{eq:cotainth}) and (\ref{eq:cotainth2}) prove (\ref{eq:teorconc2}). This finishes the proof of Theorem \ref{th:onestep}.

\vskip .5cm

We have to prove a lemma that has been used in the proof of Theorem \ref{th:onestep}, and which is nothing but a change of variables:
\begin{lemma}\label{lem:auxiliar}
In the notation of the proof of Theorem \ref{th:onestep}, we have:
\[
\mathcal{C}_0(f_t,\zeta_t)=\int_0^T\mu(h_s,\eta_s)\|(\dot h_s,\dot \eta_s)\|_{T_{(h_s,\eta_s)}\S\times\Pn}\,ds.
\]
\end{lemma}
\begin{proof}
One can just apply the change of variables formula to the change of variables $s=s(t)$ (so that $h_{s(t)}=f_t/\|f_t\|$ and $\eta_{s(t)}=\zeta_t$) and, after a long computation prove that the two integrals of the lemma are equal. However, we prefer the following geometric argument. The quantity $\mathcal{C}_0(f_t,\zeta_t)$ is by definition the length of the path $(f_t/\|f_t\|,\zeta_t)$ when $\S\times\Pn$ is endowed with the condition metric, resulting from multiplying the usual product metric by the square of the condition number at each pair $(f,z)$. Now, as a length, it is independent of the parametrization and thus $\mathcal{C}_0(f_t,\zeta_t)=\mathcal{C}_0(h_s,\eta_s)$. This is exactly the claim of the lemma.
\end{proof}

\providecommand{\bysame}{\leavevmode\hbox to3em{\hrulefill}\thinspace}
\providecommand{\MR}{\relax\ifhmode\unskip\space\fi MR }
\providecommand{\MRhref}[2]{%
  \href{http://www.ams.org/mathscinet-getitem?mr=#1}{#2}
}
\providecommand{\href}[2]{#2}


\end{document}